\documentclass[reqno]{amsart}
\usepackage{amsmath,amssymb,amsthm,color}
\usepackage{mathtools}
\mathtoolsset{showonlyrefs=true}

\usepackage{lineno}

\numberwithin{equation}{section}

\newtheorem{theorem}{Theorem}[section]
\newtheorem{definition}[theorem]{Definition}

\newtheorem{lemma}[theorem]{Lemma}
\newtheorem{proposition}[theorem]{Proposition}
\newtheorem{remark}[theorem]{Remark}

\newcommand{\R}{\mathbb{R}}

\newcommand{\N}{\mathbb{N}}
\newcommand{\ep}{\varepsilon}
\newcommand{\pa}{\partial}

\newcommand{\lr}[1]{{}\langle{}#1{}\rangle{}}

\title[Wave equation with space-dependent damping]{
Asymptotic expansion of solutions to
the wave equation with space-dependent damping
}

\author[M. Sobajima]{Motohiro Sobajima}
\address{Department of Mathematics, 
Faculty of Science and Technology, Tokyo University of Science,  
2641 Yamazaki, Noda-shi, Chiba, 278-8510, Japan}
\email{msobajima1984@gmail.com}

\author[Y. Wakasugi]{Yuta Wakasugi}
\address{Laboratory of Mathematics,
Graduate School of Engineering,
Hiroshima University,
Higashi-Hiroshima, 739-8527, Japan}
\email{wakasugi@hiroshima-u.ac.jp}

\begin{document}

\begin{abstract}
    We study the large time behavior
    of solutions to the wave equation
    with space-dependent damping
    in an exterior domain.
    We show that
    if the damping is effective,
    then the solution is
    asymptotically expanded 
    in terms of solutions of corresponding
    parabolic equations.
    The main idea to obtain the asymptotic
    expansion is the decomposition of
    the solution of the damped wave equation
    into the solution of
    the corresponding parabolic problem
    and the time derivative of the solution
    of the damped wave equation with certain
    inhomogeneous term and initial data.
    The estimate of the remainder term
    is an application of
    weighted energy methods
    with suitable supersolutions of the
    corresponding parabolic problem.
\end{abstract}

\keywords{wave equation, space-dependent damping, asymptotic expansion}
\subjclass[2020]{35L20; 35C20; 35B40}


\maketitle

\tableofcontents

\section{Introduction}
\subsection{Problem and backgrounds}
Let
$\Omega$ be an exterior domain with a  
smooth boundary $\pa\Omega$
in $\mathbb{R}^N$ with $N \ge 2$,
or $\Omega = \mathbb{R}^N$ with $N\ge 1$.
We consider the
initial-boundary value problem of
the wave equation with space-dependent damping
\begin{align}
\label{dwx}
    \left\{ \begin{array}{ll}
        \partial_t^2 u - \Delta u + a(x) \partial_t u = 0, &x \in \Omega, t>0,\\
        u(x,t) = 0, &x \in \partial \Omega, t>0,\\
        u(x,0) = u_0(x), \ 
        \partial_t u(x,0) = u_1(x),
        &x \in \Omega.
        \end{array} \right.
\end{align}
Here,
$u=u(x,t)$
is a real-valued unknown function,
and
$a(x)$ denotes the coefficient of the damping term.

We assume that
$a(x)$ is a smooth positive function on $\R^N$
having bounded derivatives
and
satisfying
\begin{align}
\label{a}
    \lim_{|x|\to \infty}
        |x|^{\alpha} a(x) = a_0
\end{align}
with some constants
$\alpha \in [0,1)$
and $a_0 > 0$.
Here, the precise meaning of \eqref{a} is
$\lim_{r\to \infty} \sup_{|x|>r} | |x|^{\alpha}a(x) -a_0 | = 0$,
that is,
the convergence is uniform in the direction.
In this case,
the damping is called effective,
and, as we will see later,
the asymptotic behavior of the solution
is closely related to a certain
corresponding parabolic problem.
Here, we remark that it is sufficient that
$a(x)$ is defined on $\bar{\Omega}$,
because we can extend it to $\R^N$
so that it has the same property as above.

The initial data
$(u_0, u_1)$
are assumed to belong to
$(H^2(\Omega)\cap H^1_0(\Omega))\times H^1_0(\Omega)$.
Then, it is known that \eqref{dwx}
admits a unique solution
\begin{align}%
\label{sol.class}
	u \in
	C([0,\infty); H^2(\Omega)) \cap
	C^1([0,\infty); H^1_0(\Omega)) \cap 
	C^2([0,\infty); L^2(\Omega))
\end{align}%
(see \cite[Theorem 2]{Ikawa}).
In our main result, we shall put stronger assumptions
on the data.

The aim of this paper is to prove the asymptotic
expansion of the solution as time tends to infinity.
In particular, we show that 
the solution $u$
is asymptotically expanded in terms of
a sequence of solutions to
corresponding parabolic equations
with certain inhomogeneous terms.

The asymptotic behavior of solutions to the damped wave equation has long history after a pioneering work by
Matsumura \cite{Ma76}.
He studied the Cauchy problem of the wave equation
with constant damping
\begin{align}\label{eq:intro:ldw}
    \partial_t^2 u - \Delta u + \partial_t u = 0,
    \quad (x,t) \in \mathbb{R}^N \times (0,\infty),
\end{align}
and applied the Fourier transform to obtain
the $L^{\infty}$ and $L^2$ estimates of solutions.
In particular, he showed that the decay rates
are the same as those of the corresponding heat equation
\begin{align}\label{eq:intro:h}
    \partial_t v - \Delta v = 0, \quad (x,t)\in \R^N\times (0,\infty).
\end{align}
After that,
the precise asymptotic profile of solutions
were studied by
Hsiao and Liu \cite{HsLi92} for
the hyperbolic conservation laws with damping,
and
by Karch \cite{Ka00}
and
by Yang and Milani \cite{YaMi00}
for \eqref{eq:intro:ldw},
and the so-called \textit{diffusion phenomena}
was proved, that is,
the solution $u$ of \eqref{eq:intro:ldw} is
asymptotically approximated by a solution of
the heat equation \eqref{eq:intro:h}
as time tends to infinity.
More detailed asymptotic behavior was studied
by many mathematicians,
and we refer the reader to
\cite{HoOg04, MaNi03, Na04, Ni03MathZ, SaWa17}
for the asymptotic behavior involving
the decomposition of solution into
the heat-part and wave-part.

For the higher order asymptotic expansions
of the Cauchy problem of \eqref{eq:intro:ldw},
Gallay and Raugel \cite{GaRa98}
determined the second order expansion when $N=1$
by the method of scaling-variables.
Moreover, by Fourier transform,
Takeda \cite{Ta15} studied the case
$N\le 3$ and obtained the expansion of any order
in terms of the Gaussian.
Michihisa \cite{Mi21} also gave another expression of
expansion for any $N\ge 1$
by the Fourier transform method.

On the other hand,
the asymptotic behavior of solutions to
the initial-boundary value problem of
the wave equation with constant damping
in an exterior domain
is also well-studied.
This problem firstly studied by
Ikehata \cite{Ik02},
and he proved the diffusion phenomena, that is, 
the asymptotic profile of solution to the extorior problem 
is given by the exterior heat semigroup with Dirichlet boundary condition.
After that, this result was extended by
Ikehata and Nishihara \cite{IkNi03},
Chill and Haraux \cite{ChHa03},
and Radu, Todorova, and Yordanov \cite{RaToYo11}
to the abstract problem
\begin{align}\label{eq:intro:abst:dw}
    u''(t) + Au(t) + u'(t) = 0, \quad t>0,
\end{align}
where $A$ is a nonnegative self-adjoint operator
in a Hilbert space.
Recently, the first author \cite{So20MA}
proved the higher order asymptotic expansion of the solution
to \eqref{eq:intro:abst:dw}
in terms of the solutions of
the corresponding first order equation.
Radu, Todorova, and Yordanov \cite{RaToYo16}
studied the diffusion phenomena for
more general abstract equation
\begin{align}
    C u''(t) + Bu(t) + u'(t) = 0, \quad t>0
\end{align}
with a nonnegative self-adjoint operator $B$
and a positive bounded operator $C$
by the method of diffusion approximation.
Nishiyama \cite{Nis16} also studied a similar problem
by the method of resolvent estimates.

We also refer the reader to
Wirth \cite{Wi04, Wid, Wi06, Wi07JDE, Wi07ADE},
Yamazaki \cite{Ya06},
and \cite{Wa19ASPM},
for the diffusion phenomena
of the wave equation with time dependent damping.

For the initial-boundary value problem of the
wave equation with space-dependent damping \eqref{dwx},
under the assumption of \eqref{a},
it is expected that the damping is classified in the following way:
\begin{itemize}
    \item (scattering) When $\alpha > 1$,
    the solution behaves like that of the wave equation without damping.
    \item (effective) When $\alpha <1$,
    the solution behaves like that of the
    corresponding heat equation.
     \item (critical) When $\alpha = 1$,
    the equation is formally invariant under
    hyperbolic scaling and the
    behavior of the solution may also depend on the constant $a_0$.
\end{itemize}
The scattering case $\alpha >1$
when $\Omega = \mathbb{R}^N \ (N\neq 2)$
or $\Omega \subset \mathbb{R}^N$ is an exterior domain
$(N\ge 3)$
was
studied by Mochizuki \cite{Mo76},
Mochizuki and Nakazawa \cite{MoNa96},
and Matsuyama \cite{Mat02},
and they proved that
there exist initial data such that
the energy of the corresponding solution
does not decay to zero,
and it approached a solution of the wave equation
without damping in the energy norm.
The case $N=2$ seems still open.

The critical case $\alpha =1$
with the assumption on $a(x)$ replaced by
$b_0 \langle x \rangle^{-1} \le a(x) \le b_1 \langle x \rangle^{-1} \ (b_0, b_1>0)$
was studied by
Ikehata, Todorova and Yordanov \cite{IkToYo13}.
They proved that,
when $\Omega = \mathbb{R}^N$ with $N\ge 3$ and the initial data are in $C_0^{\infty}(\mathbb{R}^N)$,
the energy of the solution decays as
$O(t^{-b_0})$ if $1<b_0<N$,
and
$O(t^{-N+\delta})$ with arbitrary small $\delta >0$
if $b_0 \ge N$.
Moreover, the decay rate
$O(t^{-b_0})$
when $1<b_0 < N$
is optimal under some additional assumptions
on $a(x)$.
A similar result was also obtained for the case
$N \le 2$.
This indicates that the behavior of the solution
depends on the coefficient $b_0$.

For the effective case $\alpha <1$,
Todorova and Yordanov \cite{ToYo09}
developed a weighted energy method
with an exponential-type weight function
\begin{align}
    \exp \left( m(a) \frac{A(x)}{t} \right),
\end{align}
which is a refinement of the method by
Ikehata \cite{Ik05IJPAM}.
Here, $A(x)$ is a solution of the Poission equation
$\Delta A(x) = a(x)$ satisfying
$A(x) \sim \langle x \rangle^{2-\alpha}$,
and
$m(a) = \liminf_{|x|\to \infty} \frac{a(x)A(x)}{|\nabla A(x)|^2}$.
They showed that if
$\Omega = \mathbb{R}^N \ (N\ge 1)$,
$a(x)$
is radially symmetric and satifies \eqref{a} with some
$\alpha \in [0,1)$ and $a_0>0$,
and the initial data belong to
$C_0^{\infty}(\mathbb{R}^N)$,
then we have
$m(a) = \frac{N-\alpha}{2-\alpha}$,
and the following estimates hold:
\begin{align}
    \| u(t) \|_{L^2} &\le C_\delta 
    			(1+t)^{-\frac{N-\alpha}{2(2-\alpha)}+\frac{\alpha}{2(2-\alpha)}+\delta},\\
    \| (\partial_t u(t), \nabla u(t)) \|_{L^2}
    &\le C_\delta  
    			(1+t)^{-\frac{N-\alpha}{2(2-\alpha)} -\frac{1}{2} +\delta},
\end{align}
where $\delta > 0$ is an arbitrary small loss of decay.
Radu, Todorova, and Yordanov \cite{RaToYo09, RaToYo10}
studied the energy decay of higher order derivatives
and extended the result to
more general second-order hyperbolic equations.
The assumption of the radial symmetry on $a(x)$
was removed by the authors \cite{SoWa17_AIMS}
by modifying the function $A(x)$ above.
Moreover, the authors \cite{Wa14JHDE, SoWa16_JDE, SoWa18_ADE}
proved the diffusion phenomena
in the case of $\alpha \in (-\infty,1)$ and exterior domains. 
The asymptotic profile of solution $u$ is given by the corresponding 
parabolic problem 
\begin{align}
\label{hx}
    \left\{ \begin{array}{ll}
        a(x) \partial_t V_0 - \Delta V_0 = 0, &x \in \Omega, t>0,\\
        V_0(x,t) = 0, &x \in \partial \Omega, t>0,\\
        V_0(x,0) = u_0(x) + a(x)^{-1} u_1(x),
        &x \in \Omega.
        \end{array} \right.
\end{align}

Recently, the authors \cite{SoWa19_CCM, SoWa21_JMSJ}
developed
a
different kind of weighted energy method 
applicable to a wider class of initial data 
including polynomially decaying functions. 
Roughly speaking, the suggested weight functions 
form the inverse of the self-similar solutions $\Phi_{\beta}$ of 
the equation $|x|^{-\alpha}\partial_t \Phi - \Delta \Phi = 0$
given by 
\begin{align}
    \Phi_{\beta} (x,t) = t^{-\beta} \varphi_{\beta} (\xi(x,t)),
\end{align}
where $\beta \in [0,\frac{N-\alpha}{2-\alpha})$ is a parameter, 
\begin{align}
    \varphi_{\beta} (z)
    = e^{-z}
    M \left( \frac{N-\alpha}{2-\alpha} - \beta, \frac{N-\alpha}{2-\alpha}; z \right),\quad
    \xi(x,t) = \frac{|x|^{2-\alpha}}{(2-\alpha)^2t}
\end{align}
with 
the Kummer confluent hypergeometric function 
$M(b,c;z)$ (see Section 2 for the precise definition). 
Moreover, the relation between the order of the weight
of initial data and the decay rates of the solution
was revealed. 
It is worthly noticing that these weight functions 
have a polynomial growth which enables us 
to take initial data having a polynomial decay, 
and the endpoint $\beta=\frac{N-\alpha}{2-\alpha}$ 
provides the exponential type solution 
\[
\Phi_{\frac{N-\alpha}{2-\alpha}}(x,t)
=t^{-\frac{N-\alpha}{2-\alpha}}\exp\left(-\frac{|x|^{2-\alpha}}{(2-\alpha)t}\right)
\]
which corresponds to the exponential type weight function introduced in \cite{ToYo09}. 

As mentioned above, the sharp decay estimates of solutions and
the diffusion phenomena
for the effective case $\alpha < 1$
is now known very well.
In contrast,
the higher order asymptotic expansion of the solution remains open.

Here, we mention a result by
Orive, Zuazua, and Pazoto \cite{OrZuPa01}
and Joly and Royer \cite{JoRo18}
for periodic and asymptotically periodic coefficient cases 
from different aspects.
For the exterior problem with decaying damping such as \eqref{dwx}, 
it seems difficult to apply the Fourier analysis 
which is the strong tool for the whole space case, 
and to apply the spectral analysis because of 
the appearance of unbounded diffusion operators 
and non-comutativity. 

In the present paper we introduce a new method (inspired by \cite{So20MA})
to reach the asymptotic expansion 
in terms of solutions of the corresponding 
parabolic equation with certain inhomogeneous terms
(see a description of the idea in Subsection 1.3).

\subsection{Main result}

The following is our main result
which desribes the higher order asymptotic expansion
of the solution $u$
of \eqref{dwx}. 
\begin{theorem}\label{thm:asym}
Let $n$ be a nonnegative integer.
Assume \eqref{a}
for some 
$\alpha \in [0,1)$
and
$a_0>0$.
If 
$n+1 < \frac{N-\alpha}{2\alpha}$
and
$\lambda \in [\frac{2\alpha}{2-\alpha}(n+1), \frac{N-\alpha}{2-\alpha})$,
then
there exist a positive integer
$s = s(n)$
and a constant
$m = m(n,\alpha,\lambda) >0$
such that the following holds:
Suppose the initial data
$u_0$ and $u_1$ satisfy
\begin{align}
\label{ass:thm:asym}
    u_0 \in H^{s+1,m}(\Omega)\cap H_0^{s,m}(\Omega),\quad
    u_1 \in H^{s,m}_0(\Omega).
\end{align}
Then there exist profiles $\widetilde{V}_1 \ldots, \widetilde{V}_n\in C([0,\infty);L^2(\Omega))$ 
and a positive constant $C$ such that 
the solution $u$ of \eqref{dwx} satisfies 
\begin{align}\label{eq:thm:asymp}
    \left\| u(t) -V_0(t) -\sum_{j=1}^{n} \widetilde{V}_j(t)\right\|_{L^2(\Omega)} 
    \le C (1+t)^{-\frac{\lambda}{2} - \frac{(2n+1)(1-\alpha)}{2-\alpha}
    +\frac{\alpha}{2(2-\alpha)}}
\end{align}
for $t > 0$, where $V_0$ is given by \eqref{hx}. 
Moreover, the profiles are successively determined as $\widetilde{V}_j=\pa_t^jV_{j}$
with the unique solutions $V_{j}$ of 
\begin{align}\label{eq:intro:V_j}
     \left\{ \begin{array}{ll}
       a(x) \partial_t V_j - \Delta V_j = - \partial_t V_{j-1}, &x \in \Omega, t>0,\\
       V_j(x,t) = 0, & x \in \partial \Omega, t>0,\\
        V_j(x,0) = - (-a(x))^{-j-1} u_1(x),
        &x \in \Omega.
        \end{array} \right.
\end{align}
for $j=1, \ldots, n$. 
\end{theorem}

\begin{remark}
For each $V_j$
$(j=0, 1, \ldots, n)$,
we also have
\begin{align}
    \| \partial_t^j V_j (t) \|_{L^2(\Omega)}
    \le C (1+t)^{-\frac{\lambda}{2} - \frac{2j(1-\alpha)}{2-\alpha}
    +\frac{\alpha}{2(2-\alpha)}}
\end{align}
(see Section 6). We remark that
when $a(x) \equiv 1$,
the expansion in Theorem \ref{thm:asym} coincides with
the known result \cite{So20MA}. 
\end{remark}

\begin{remark}
One can also represent the profiles $\widetilde{V}_j$ for $j=1,\ldots n$ in terms of 
the semigroup $e^{tL}$ generated by $L=a^{-1}\Delta$. 
For instance, the second profile $\widetilde{V}_1$ can be written as 
\[
\widetilde{V}_1(t)
=
-Le^{tL}[a^{-2}u_1]
-a^{-1}Le^{tL}[u_0+a^{-1}u_1]
-\int_0^tLe^{(t-s)L}\Big[a^{-1}Le^{sL}[u_0+a^{-1}u_1]\Big]\,ds.
\]
If $a\equiv 1$, then $e^{t\Delta}$ and $a^{-1}$ commutes, and therefore, 
the above description can be simplified to 
$
\widetilde{V}_1(t)
=-\Delta e^{t\Delta}u_1-\Delta (1+t\Delta)e^{t\Delta }[u_0+u_1]$
as in \cite{So20MA}, 
but the semigroup $e^{tL}$ and $a^{-1}$ do not commute in general. 
To avoid such a complicated situation, we have chosen 
the parabolic equations \eqref{eq:thm:asymp} for the determinination of the profiles $\widetilde{V}_j$.
\end{remark}

\begin{remark}
\textup{(i)}
About the explicit values of $s = s(n)$ and $m=m(n,\alpha,\lambda)$
in Theorem \ref{thm:asym},
a rough computation shows that we can take
$s = 5(n+1)$ and $m= (\lambda+2n+1)\frac{2-\alpha}{2} + (6n^2 + 14n+8)\alpha$.
However, we omit the detailed computation, and
do not discuss the optimality of them here.

\noindent
\textup{(ii)}
If $u_0, u_1 \in C_0^{\infty}(\Omega)$,
then the assumptions on the initial data
of Theorem \ref{thm:asym} are
automatically fulfilled.
\end{remark}

\subsection{A rough descripsion of strategy}

By the previous studies
\cite{Wa14JHDE, SoWa16_JDE, SoWa21_JMSJ},
the solution of \eqref{hx}
is known to be the first asymptotic profile of
the solution of \eqref{dwx}.
To investigate the asymptotic behavior of solution
to \eqref{dwx}, we follow the idea of
\cite{So20MA}.
First, the fact that
$V_0$ is the first asymptotic profile
implies that
$u - V_0$
is a remainder term.
In \cite{So20MA}, it is found that
the remainder term
$u-V_0$
can be expressed as the
time derivative of the solution of
the damped wave equation with
a
certain inhomogeneous term.
More precisely,
let $U_1$ be the solution of
\begin{align*}
    \left\{ \begin{array}{ll}
        \partial_t^2 U_1 - \Delta U_1 + a(x) \partial_t U_1 = - \partial_t V_0, &x \in \Omega, t>0,\\
        U_1(x,t) = 0, &x \in \partial \Omega, t>0,\\
        U_1(x,0) = 0, \ 
        \partial_t U_1(x,0) = - a(x)^{-1} u_1(x),
        &x \in \Omega.
        \end{array} \right.
\end{align*}
Next, we have the decomposition
$u = V_0 + \partial_t U_1$ (see Lemma \ref{lem_eq_U}).
Then we further consider the asymptotic profile of $U_1$. 
By experience, 
it is natural to choose $V_1$ via \eqref{eq:thm:asymp} with $n=1$ 
($V_1$ and $U_1$ has the same inhomogeneous term in respective equations). 
Then, in a similar way $U_1$ can also be also decomposed as $U_1=V_1+\pa_tU_2$ 
with the second auxiliary function $U_2$ via
\begin{align*}
    \left\{ \begin{array}{ll}
        \partial_t^2 U_{2} - \Delta U_{2} + a(x) \partial_t U_{2} = - \partial_t V_1, &x \in \Omega, t>0,\\
        U_{2}(x,t) = 0, & x \in \partial \Omega, t>0,\\
        U_{2}(x,0) = 0, \ 
        \partial_t U_{2}(x,0) = (-a(x))^{-2} u_1(x),
        &x \in \Omega.
        \end{array} \right.
\end{align*}
The relation 
\[
u=V_0+\pa_tU_1=V_0+\pa_tV_1+\pa_t^2U_2
\]
can be expected to determine the second expansion. 
Continuously, using the $(n+1)$-th auxiliary function $U_{n+1}$ given by 
\begin{align}
\label{eq:intro:U_n+1}
    \left\{ \begin{array}{ll}
        \partial_t^2 U_{n+1} - \Delta U_{n+1} + a(x) \partial_t U_{n+1} = - \partial_t V_n, &x \in \Omega, t>0,\\
        U_{n+1}(x,t) = 0, & x \in \partial \Omega, t>0,\\
        U_{n+1}(x,0) = 0, \ 
        \partial_t U_{n+1}(x,0) = (-a(x))^{-n-1} u_1(x),
        &x \in \Omega,
        \end{array} \right.
\end{align}
one can obtain the relation 
\begin{align}
\label{eq:intro:u_decomp}
    u = V_0 + \partial_t V_1
        + \partial_t^2 V_2 + \cdots
        + \partial_t^n V_n
        + \partial_t^{n+1} U_{n+1}.
\end{align}
More precise discussion will be given in
Section 3.
Note that even if the initial data $(u_0,u_1)$ are compactly supported, 
$V_0,\ldots, V_n$ and $U_{n+1}$ do not have 
compact supports in general.
Therefore the finite propagatoin property does not work in this situation.
Applying a weighted energy method developed by
the authors' previous papers
\cite{So19DIE, SoWa21_JMSJ},
we prove that $\partial_t^{n+1} U_{n+1}$ decays faster than the other terms in \eqref{eq:intro:u_decomp},
and this implies that the solution
$u$
is asymptotically expanded by
the sum of
$V_0, \partial_t V_1, \ldots, \partial_t^n V_n$.

\subsection{Construction of the paper}
This paper is constructed as follows.
In the next section,
we prepare the weight functions
used in the energy method in subsequent sections.
In section 3,
we state the well-posedness and regularity
of solutions of the problem \eqref{dwx}
and discuss the validity of the 
decomposition \eqref{eq:intro:u_decomp} (formally explained in Subsection 1.3) 
in a suitable weighted Sobolev space.
In Section 4,
we discuss the weighted energy estimates for
the corresponding parabolic equations.
In Section 5,
we prove the weighted energy estimates for
the damped wave equation \eqref{dwx} with an inhomogeneous term.
Finally, in Section 6,
we complete the proof of Theorem \ref{thm:asym}
by adapting the energy estimates
prepared in Sections 4 and 5
to the original problem \eqref{dwx}.

\subsection{Notations}
We finish this section with some notations used throughout this paper.
The letter $C$ indicates a generic positive constant,
which may change from line to line.
We also express constants by
$C(\ast, \ldots, \ast)$,
which means this constant depends on the parameters in the parenthesis.
The symbol
$f \lesssim g$
stands for
$f\le Cg$
holds with some constant
$C>0$,
and
$f \sim g$ means
both $f \lesssim g$ and $g \lesssim f$ hold.

We denote
$\langle x \rangle = \sqrt{1+|x|^2}$
for $x \in \mathbb{R}^n$.
Let
$L^2(\Omega)$
be the usual Lebesgue space
with the norm
\begin{align}
    \| f \|_{L^2(\Omega)}
    = \left( \int_{\Omega} |f(x)|^2 \,dx \right)^{1/2},
\end{align}
and
$C_0^{\infty}(\Omega)$
stands for the space of
infinitely differentiable functions with
compact support in $\Omega$.
For a nonnegative integer $k$ and $m \in \mathbb{R}$,
we introduce the weighted Sobolev spaces by
\begin{align}
\label{eq:L2}
    H^{k,m} (\Omega) &= \{ f: \Omega \to \mathbb{R} ;  \langle x \rangle^m \partial^{\alpha}_x f \in L^2(\Omega)
            \, \mbox{for any}\,
            \alpha \in \mathbb{Z}_{\ge 0}^N
            \, \mbox{with}\, |\alpha| \le k \},\\
    \| f \|_{H^{k,m}(\Omega)}
    &=
    \sum_{|\alpha|\le k} \| \langle x \rangle^m \partial_x^{\alpha} f \|_{L^2(\Omega)},
\end{align}
where we used the notion of multi-index
and the derivatives are in the sense of
distribution.
When
$m=0$,
we denote $H^{k}(\Omega) = H^{k,0}(\Omega)$
for short.
Also,
$H^{k,m}_0(\Omega)$
is the completion of
$C_0^{\infty}(\Omega)$
with respect to the norm
$\| f \|_{H^{k,m}(\Omega)}$.

\section{Preliminaries}
\subsection{Weight functions}

Throughout this section, we slightly generalize the conditions on
$a(x)$
and assume that
$a(x)$
is a smooth positive function on
$\mathbb{R}^N$
satisfying
\begin{align}
\label{a:sec2}
    \lim_{|x|\to \infty}
        |x|^{\alpha} a(x) = a_0
\end{align}
with some constants
$\alpha\in (-\infty,\min\{2,N\})$
and $a_0 > 0$.

We prepare weight functions constructed in \cite{SoWa21_JMSJ}.
First, we introduce a suitable approximate solution of
the Poisson equation
$\Delta A(x) = a(x)$.

\begin{lemma}[{\cite[Lemma 2.1]{SoWa17_AIMS}, \cite[Lemma 3.2]{SoWa21_JMSJ}}]\label{lem_A_ep}
Let $a(x)$ be a smooth positive function on $\R^N$
satisfying the condition \eqref{a:sec2} with some constants
$\alpha\in (-\infty,\min\{2,N\})$
and $a_0 > 0$.
Then for every $\ep \in (0,1)$, 
there exist a function $A_\ep\in C^2(\R^N)$ 
and positive constants $c_\ep$ and $C_\ep$ such that
\begin{align}
\label{A1}
	&(1-\ep)a(x)\leq \Delta A_\ep (x) \leq (1+\ep) a(x),\\
\label{A2}
	&c_\ep \lr{x}^{2-\alpha} \leq A_\ep(x) \leq C_\ep \lr{x}^{2-\alpha},\\
\label{A3}
	&\frac{|\nabla A_\ep(x)|^2}{a(x)A_\ep(x)}\leq \frac{2-\alpha}{N-\alpha}+\ep
\end{align}
hold for $x\in \mathbb{R}^N$.
\end{lemma}

\begin{remark}
The above type function $A_{\ep}(x)$ was firstly introduced
by Ikehata \cite{Ik05IJPAM}, Todorova and Yordanov \cite{ToYo09}, and Nishihara \cite{Ni10}.
In particular, in \cite{ToYo09},
a solution of
$\Delta A(x) = a(x)$,
that is,
the equation obtained by taking
$\varepsilon = 0$
in \eqref{A1},
was applied for weighted energy estimates
for the damped wave equation \eqref{dwx} with radially symmetric $a(x)$.
Lemma \ref{lem_A_ep} is a refinement of the method of \cite{ToYo09} to remove
the assumption of radial symmetry on $a(x)$.
\end{remark}

The following definitions
are connected to the supersolution of
$a(x)v_t - \Delta v = 0$
constructed in \cite{SoWa21_JMSJ},
which plays a crucial role to obtain
several estimates verifying asymptotic expansion.

\begin{definition}[Kummer's confluent hypergeometric functions]
\label{def:M}
For
$b,c \in \mathbb{R}$ with $-c \notin \mathbb{N} \cup \{0\}$,
Kummer's confluent hypergeometric function of first kind is defined by
\begin{align*}
	M(b, c; s) = \sum_{n=0}^{\infty} \frac{(b)_n}{(c)_n} \frac{s^n}{n!}, \quad s\in [0,\infty),
\end{align*}
where $(d)_n$ is the Pochhammer symbol defined by 
$(d)_0 = 1$ and $(d)_n = \prod_{k=1}^n (d+k-1)$ for $n\in\N$; 
note that when $b=c$, $M(b,b;s)$ coincides with $e^s$.
\end{definition}

\begin{definition}\label{phi.beta}
\noindent
{\rm (i)}
    For $\varepsilon \in (0,1/2)$, we define
\begin{align}
\label{gammatilde}
	\widetilde{\gamma}_\ep=\left(\frac{2-\alpha}{N-\alpha}+2 \ep\right)^{-1}, \quad 
	\gamma_\ep=(1-2\ep)
	\widetilde{\gamma}_\ep.
\end{align}

\noindent
{\rm (ii)}
For
$\beta \ge 0$
and
$\varepsilon \in (0,1/2)$,
define 
\[
	\varphi_{\beta,\ep}(s)=e^{-s}M\left(\gamma_\ep-\beta, \gamma_\ep; s\right),
	\quad s\geq 0.
\]
\end{definition}

\begin{remark}
{\rm (i)}
We slightly modify the definition of
$\tilde{\gamma}_{\varepsilon}$
and
$\gamma_{\varepsilon}$
from those of \cite{SoWa21_JMSJ}
in order to gain a positive term
in the right-hand side of Proposition \ref{prop:super-sol} (iv).
This modification enables us to
unify the proof of energy estimates
for the case
$N=1$ and $N \ge 2$
(see Sections 4 and 5).

\noindent
{\rm (ii)}
We note that
$\varphi_{\beta,\varepsilon}(s)$
is a unique (modulo constant multiple) solution of
\begin{align}\label{eq:varphi}
    s \varphi''(s) + (\gamma_{\varepsilon} + s)\varphi'(s) + \beta \varphi(s) = 0
\end{align}
with bounded derivative near $s=0$.
\end{remark}

\begin{lemma}\label{lem:phi.beta}
The function $\varphi_{\beta,\ep}$ defined in Definition \ref{phi.beta} satisfies the following properties.
\begin{itemize}
    \item[(i)]
    If $0\le \beta < \gamma_{\ep}$, then
    $\varphi_{\beta,\ep}(s)$ satisfies the estimates
    \begin{align}
        k_{\beta,\ep} (1+s)^{-\beta} \le \varphi_{\beta,\ep} (s) \le K_{\beta,\ep} (1+s)^{-\beta}
    \end{align}
    with some constants
    $k_{\beta,\ep}, K_{\beta,\ep} > 0$.
    \item[(ii)]
    For every $\beta \ge 0$,
    the estimate
    \begin{align}
        |\varphi_{\beta,\ep}(s)| \le K_{\beta,\ep}(1+s)^{-\beta}
    \end{align}
    holds with some constant $K_{\beta,\ep} > 0$.
    \item[(iii)]
    For every $\beta \ge 0$,
    $\varphi_{\beta,\ep}(s)$ and $\varphi_{\beta+1,\ep}(s)$
    satisfy the recurrence relation
    \begin{align}
        \beta \varphi_{\beta,\ep}(s) + s \varphi_{\beta,\ep}'(s)
        = \beta \varphi_{\beta+1,\ep}(s).
    \end{align}
    \item[(iv)]
    If $0\le \beta < \gamma_{\ep}$, then we have
    \begin{align}
        \varphi_{\beta,\ep}'(s) &= - \frac{\beta}{\gamma_{\ep}} e^{-s} M(\gamma_{\ep}-\beta, \gamma_{\ep} + 1; s) \le 0,\\
        \varphi_{\beta,\ep}''(s) &= \frac{\beta(\beta+1)}{\gamma_{\ep}(\gamma_{\ep}+1)} e^{-s} M(\gamma_{\ep}-\beta, \gamma_{\ep} + 2; s) \ge 0.
    \end{align}
    \item[(v)]
    If
    $0\le \beta < \gamma_{\varepsilon}$,
    then
    $\varphi_{\beta,\varepsilon}'$
    satisfies
    \begin{align}
        -\varphi_{\beta,\varepsilon}'(s)
        \ge
        k_{\beta,\varepsilon} (1+s)^{-\beta-1}
    \end{align}
    holds with some constant
    $k_{\beta,\varepsilon} > 0$.
\end{itemize}
\end{lemma}
\begin{proof}
The proof of the assertions
(i)--(iv) are completely the same as that of
\cite[Lemma 3.5]{SoWa21_JMSJ},
and we omit the detail.
The property (v) follows from
the expression in (vi) and
the fact
$M(\gamma_{\varepsilon}-\beta, \gamma_{\varepsilon} + 1; s)
\sim \frac{\Gamma(\gamma_{\varepsilon}+1)}{\Gamma(\gamma_{\varepsilon}-\beta)}s^{-\beta-1}e^s$
as
$s\to \infty$
(see, for example, \cite[Lemma 2.2 (ii)]{SoWa21_JMSJ} or
\cite[p.192, (6.1.8)]{BeWo10}).
\end{proof}

Here, we give a family of supersolutions of
$a(x)v_t - \Delta v = 0$,
which we use later.

\begin{definition}\label{phi.beta.ep}
For
$\beta \ge 0$
and
$(x,t) \in \mathbb{R}^N \times [0,\infty)$,
we define 
\[
	\Phi_{\beta,\varepsilon}(x,t; t_0)=(t_0+t)^{-\beta}\varphi_{\beta,\varepsilon}(z), 
	\quad 
	z=\frac{\widetilde{\gamma}_\varepsilon A_\varepsilon(x)}{t_0+t},
\]
where
$\ep \in (0,1/2)$,
$\widetilde \gamma_{\ep}$ is the constant given in \eqref{gammatilde},
$t_0 \ge 1$,
$\varphi_{\beta,\ep}$
is the function defined by Definition \ref{phi.beta},
and
$A_{\ep}(x)$ is the function constructed in Lemma \ref{lem_A_ep}.
\end{definition}

For $t_0 \ge 1$ and
$(x,t)\in \mathbb{R}^N \times [0,\infty)$,
we also define
\begin{align}
\label{psi}
	\Psi(x,t; t_0) :=
		t_0 + t + A_{\varepsilon}(x).
\end{align}

\begin{proposition}\label{prop:super-sol}
The function
$\Phi_{\beta,\ep}(x,t;t_0)$
defined in Definition \ref{phi.beta.ep}
satisfies the following properties:
\begin{itemize}
\item[(i)]
For every $\beta \ge 0$, we have
\begin{align*}
	\partial_t \Phi_{\beta,\varepsilon}(x,t;t_0) = -\beta \Phi_{\beta+1,\varepsilon}(x,t;t_0)
\end{align*}
for any $(x,t) \in \mathbb{R}^N \times [0,\infty)$.
\item[(ii)]
If $\beta \ge 0$, then
there exists a constant $C_{\alpha,\beta,\varepsilon} > 0$ such that
\begin{align*}
	|\Phi_{\beta,\varepsilon}(x,t;t_0) | \le C_{\alpha,\beta,\varepsilon} \Psi (x,t; t_0)^{-\beta}
\end{align*}
for any $(x,t) \in \mathbb{R}^N \times [0,\infty)$.
\item[(iii)]
If $\beta \in [0, \gamma_{\varepsilon})$,
then there exists a constant $c_{\alpha,\beta,\varepsilon} > 0$ such that
\begin{align*}
	\Phi_{\beta,\varepsilon}(x,t;t_0) \ge c_{\alpha,\beta,\varepsilon} \Psi (x,t; t_0)^{-\beta}
\end{align*}
for any $(x,t) \in \mathbb{R}^N \times [0,\infty)$.
\item[(iv)]
For every $\beta \ge 0$,
there exists a constant $c_{\alpha,\beta,\varepsilon}>0$ such that
\begin{align*}
	a(x)\pa_t\Phi_{\beta,\ep}(x,t;t_0)-\Delta \Phi_{\beta,\ep}(x,t;t_0)
	\ge c_{\alpha,\beta,\varepsilon} a(x) \Psi(x,t;t_0)^{-\beta-1}.
\end{align*}
for any
$(x,t) \in \mathbb{R}^N \times [0,\infty)$.
\end{itemize}
\end{proposition}
\begin{proof}
The properties (i)--(iii) are the same as
\cite[Lemma 3.8]{SoWa21_JMSJ} and \cite[Lemma 5.1]{SoWa21_JMSJ}.
Thus, we omit the detail.
For (iv), we put
$z = \tilde{\gamma}_{\varepsilon} A_{\varepsilon}(x)/(t_0+t)$
and compute
\begin{align}
    &(t_0+t)^{\beta+1}
    \left(
    a(x) \partial_t \Phi_{\beta,\varepsilon} (x,t;t_0) - \Delta \Phi_{\beta,\varepsilon} (x,t;t_0)
    \right)\\
    &=
    -a(x)
    \left( \beta \varphi_{\beta,\varepsilon}(z) + z \varphi_{\beta,\varepsilon}'(z)
    + \tilde{\gamma}_{\varepsilon} \frac{\Delta A_{\varepsilon}(x)}{a(x)} \varphi_{\beta,\varepsilon}'(z) + \tilde{\gamma}_{\varepsilon} \frac{|\nabla A_{\varepsilon}(x)|^2}{a(x)A_{\varepsilon}(x)} z \varphi_{\beta,\varepsilon}''(z) \right).
\end{align}
Using the equation \eqref{eq:varphi} with \eqref{gammatilde},
we rewrite the right-hand side as
\begin{align}
    &\tilde{\gamma}_{\varepsilon} a(x) 
    \left( 1 - 2\varepsilon - \frac{\Delta A_{\varepsilon}(x)}{a(x)} \right) \varphi_{\beta,\varepsilon}'(z)
    + a(x)
    \left( 1 - \tilde{\gamma}_{\varepsilon} \frac{|\nabla A_{\varepsilon}(x)|^2}{a(x)A_{\varepsilon}(x)} \right) z \varphi_{\beta,\varepsilon}''(z).
\end{align}
By \eqref{A1} and \eqref{A3} in Lemma \ref{lem_A_ep},
we have
\begin{align}
    &1-2\varepsilon - \frac{\Delta A_{\varepsilon}(x)}{a(x)}
    \le - \varepsilon,\\
    &1 - \tilde{\gamma}_{\varepsilon} \frac{|\nabla A_{\varepsilon}(x)|^2}{a(x)A_{\varepsilon}(x)}
    \ge
    \varepsilon \left( \frac{2-\alpha}{N-\alpha} + 2\varepsilon \right)^{-1} > 0.
\end{align}
Combining them to the properties (iv) and (v)
in Lemma \ref{lem:phi.beta},
we conclude
\begin{align}
    a(x) \partial_t \Phi_{\beta,\varepsilon} (x,t;t_0) - \Delta \Phi_{\beta,\varepsilon} (x,t;t_0)
    &\ge 
    - \varepsilon \tilde{\gamma}_{\varepsilon} a(x) (t_0+t)^{-\beta-1}  \varphi_{\beta,\varepsilon}' \left( \frac{\tilde{\gamma}_{\varepsilon} A_{\varepsilon}(x)}{t_0+t} \right) \\
    &\ge 
    C a(x) (t_0+t)^{-\beta-1} \left( 1 + \frac{\tilde{\gamma}_{\varepsilon} A_{\varepsilon}(x)}{t_0+t} \right)^{-\beta-1} \\
    &\ge 
    C a(x)  \left(t_0 + t + A_{\varepsilon}(x) \right)^{-\beta-1} \\
    &=
    C a(x) \Psi(x,t;t_0)^{-\beta-1},
\end{align}
which completes the proof.
\end{proof}

Finally, we prepare a useful lemma
for our weighted energy method.
\begin{lemma}[{\cite[Lemma 2.5]{So19DIE}}]\label{lem.deltaphi}
Let $\Phi \in C^2(\overline{\Omega})$
be a positive function and
let $\delta \in (0, 1/2)$.
Then, for any
$u\in H^2(\Omega) \cap H^1_0(\Omega)$,
we have
\begin{align*}
	\int_{\Omega} u \Delta u \Phi^{-1+2\delta}\,dx
		&\le - \frac{\delta}{1-\delta} \int_{\Omega} |\nabla u|^2 \Phi^{-1+2\delta}\,dx
			+ \frac{1-2\delta}{2} \int_{\Omega} u^2 (\Delta \Phi) \Phi^{-2+2\delta} \,dx,
\end{align*}
provided that the right-hand side is finite.
\end{lemma}

\section{Justification of the decomposition}
In this section, 
we justify the decomposition
\[
u=\sum_{j=0}^n\pa_t^jV_j+\pa_t^{n+1}U_{n+1}
\]
which is explained in Subsection 1.3. 
Here we need to clarify 
existence, uniqueness and also an expected regularity 
of respective components $V_0$, \ldots $V_n$ and $U_{n+1}$. 
Therefore we discuss it in the following way:
we first prepare the well-posedness of
the initial-boundary value problem of the damped wave equation.
Next, we show a key decomposition lemma
which states that a solution of the damped wave equation
can be decomposed into a solution of
the corresponding parabolic equation and
the derivative of a solution of the damped wave equation
with another inhomogeneous term.
Finally, using the decomposition lemma repeatedly,
we explain how the higher order asymptotic profiles are
determined.

\subsection{Well-posedness and regularity of solutions for the damped wave equation}

We consider the initial-boundary value problem
of the damped wave equation
with a general inhomogeneous term
\begin{align}\label{eq:gDW}
     \left\{ \begin{array}{ll}
        \partial_t^2 w - \Delta w + a(x) \partial_t w = F, &x \in \Omega, t>0,\\
        w(x,t) = 0, & x \in \partial \Omega, t>0,\\
        w(x,0) = w_0(x), \ 
        \partial_t w(x,0) = w_1(x),
        &x \in \Omega.
        \end{array} \right.
\end{align}
We first prepare the well-posedness and the regularity of solutions for
\eqref{eq:gDW}.

We recall the following well-posedness result
by Ikawa \cite{Ikawa}.
\begin{theorem}[{\cite[Theorem 1]{Ikawa}}] \label{thm:Ikawa:1}
For any
$(w_0, w_1) \in (H^2(\Omega)\cap H^1_0(\Omega)) \times H^1_0(\Omega)$
and
$F \in C^1([0,\infty);L^2(\Omega))$,
there exists a unique solution
\begin{align}
\label{eq:sol:gDW:tr}
    w \in C([0,\infty); H^2(\Omega))
        \cap C^1([0,\infty); H^1_0(\Omega))
        \cap C^2([0,\infty); L^2(\Omega))
\end{align}
of \eqref{eq:gDW}.
\end{theorem}
By applying the above theorem to
$\langle x \rangle^m w$,
$(\langle x \rangle^m w_0, \langle x \rangle^m w_1)$,
and
$\langle x \rangle^m F$,
we have the well-posedness of
the problem \eqref{eq:gDW}
in weighted Sobolev spaces.
\begin{theorem}\label{thm:WP:gDW}
For any
$(w_0, w_1) \in
(H^{2,m}(\Omega) \cap H^{1,m}_0(\Omega))
\times H^{1,m}_0(\Omega)$
and
$F \in C^1([0,\infty); H^{0,m}(\Omega))$,
there exists a unique solution
\begin{align}
\label{eq:sol:gDW:tr:m}
    w \in C([0,\infty); H^{2,m}(\Omega))
        \cap C^1([0,\infty); H^{1,m}_0(\Omega))
        \cap C^2([0,\infty); H^{0,m}(\Omega))
\end{align}
of \eqref{eq:gDW}.
\end{theorem}
Next, we discuss the regularity of the solution.
We first recall the following regularity theorem
by Ikawa \cite{Ikawa}:
\begin{theorem}[{\cite[Theorem 2]{Ikawa}}]\label{thm:reg:gDW:ikw}
Let $k \ge 1$ be an integer and let
$m\ge 0$,
$w_0 \in H^{k+2}(\Omega)$,
$w_1 \in H^{k+1}(\Omega)$,
and
$F \in \bigcap_{j=0}^{k} C^{j+1}([0,\infty); H^{k-j}(\Omega))$.
We successively define
\begin{align}
    w_p = \Delta w_{p-2} - a(x) w_{p-1} + \partial_t^{p-2} F(x,0)
\end{align}
for $p = 2,\ldots, k+1$,
and assume the $k$-th order compatibility condition
\begin{align}
    (w_p, w_{p+1} ) \in (H^{2}(\Omega) \cap H^{1}_0(\Omega)) \times H^{1}_0(\Omega)
\end{align}
for $p = 0,1,\ldots,k$.
Then, the solution $w$ to \eqref{eq:gDW}
obtained by Theorem \ref{thm:Ikawa:1}
belongs to
\begin{align}
    C([0,\infty);H^{k+2}(\Omega))
    \cap
    \left( \bigcap_{j=1}^{k+1}
        C^{k+2-j} ([0,\infty); H^{j}_0(\Omega)) \right)
    \cap
    C^{k+2}([0,\infty);L^2(\Omega)).
\end{align}
\end{theorem}
From the above theorem and the same argument as
Theorem \ref{thm:WP:gDW},
we have the following regularity
theorem in weighted Sobolev spaces.
\begin{theorem}\label{thm:reg:gDW}
Let $k \ge 1$ be an integer and let
$m\ge 0$,
$w_0 \in H^{k+2,m}(\Omega)$,
$w_1 \in H^{k+1,m}(\Omega)$,
and
$F \in \bigcap_{j=0}^{k} C^{j+1}([0,\infty); H^{k-j,m}(\Omega))$.
We successively define
\begin{align}
    w_p = \Delta w_{p-2} - a(x) w_{p-1} + \partial_t^{p-2} F(x,0)
\end{align}
for $p = 2,\ldots, k+1$,
and assume the $k$-th order compatibility condition
\begin{align}
    (w_p, w_{p+1} ) \in (H^{2,m}(\Omega) \cap H^{1,m}_0(\Omega)) \times H^{1,m}_0(\Omega)
\end{align}
for $p = 0,1,\ldots,k$.
Then, the solution $w$ to \eqref{eq:gDW}
obtained by Theorem \ref{thm:WP:gDW}
belongs to
\begin{align}
    C([0,\infty);H^{k+2,m}(\Omega))
    \cap
    \left( \bigcap_{j=1}^{k+1}
        C^{k+2-j} ([0,\infty); H^{j,m}_0(\Omega)) \right)
    \cap
    C^{k+2}([0,\infty);H^{0,m}(\Omega)).
\end{align}
\end{theorem}

\subsection{Regularity of solutions for the corresponding heat equation}
Following our previous study
\cite[section 2]{SoWa16_JDE},
we prepare the well-posedness
and regularity of solutions for
the initial-boundary problem of
the corresponding heat equation
with a general inhomogeneous term
\begin{align}\label{eq:h:inhom}
    \left\{ \begin{array}{ll}
       a(x) \partial_t v - \Delta v = G, &x \in \Omega, t>0,\\
       v(x,t) = 0, & x \in \partial \Omega, t>0,\\
        v(x,0) = v_0(x),
        &x \in \Omega.
        \end{array} \right.
\end{align}
Let
$d\mu = a(x) dx$
and we define
\begin{align}
    L^2_{d\mu} (\Omega)
    &=
    \left\{ 
    f \in L^2_{\textrm{loc}}(\Omega); \,
    \| f \|_{L^2_{d\mu}}=
    \left( \int_{\Omega} |f(x)|^2 \,d\mu \right)^{1/2} < \infty
    \right\},\\
    (f,g)_{L^2_{d\mu}}
    &:= \int_{\Omega} f(x) g(x) \,d\mu.
\end{align}
The operator
$-a(x)^{-1} \Delta$
is formally symmetric in
$L^2_{d\mu}(\Omega)$,
and its bilinear closed form is
defined by
\begin{align}
    \mathfrak{a}(u,v)
    &= \int_{\Omega} \nabla u(x) \cdot \nabla v(x) \,dx,\\
    D(\mathfrak{a})
    &=
    \left\{
    u \in L^2_{d\mu}(\Omega) \cap \dot{H}^1(\Omega) ;\,
    \int_{\Omega} \frac{\partial u}{\partial x_j} \varphi \,dx
    = - \int_{\Omega} u \frac{\partial \varphi}{\partial x_j} \,dx
    \ \text{for all} \ 
    \varphi \in C_0^{\infty}(\mathbb{R}^N)
    \right\}.
\end{align}
From \cite{SoWa16_JDE},
we have the Friedrichs extension
$-L$
of the operator
$-a(x)^{-1}\Delta$ in
$L^2_{d\mu}(\Omega)$.

\begin{lemma}[{\cite[Lemma 2.2]{SoWa16_JDE}}]
The operator $-L$ in $L^2_{d\mu}(\Omega)$
defined by
\begin{align}
    D(L) &= 
    \left\{
    u \in D(\mathfrak{a}) ;\,
    \exists f \in L^2_{d\mu}(\Omega)
    \ \text{s.t.}\ 
    \mathfrak{a}(u,v) = (f,v)_{L^2_{d\mu}}
    \ \text{for any} \ 
    v \in D(\mathfrak{a})
    \right\},\\
    -L u &= f
\end{align}
is nonnegative and selfadjoint in
$L^2_{d\mu}(\Omega)$.
Therefore,
$L$ generates an analytic semigroup
$T(t)$
on $L^2_{d\mu}(\Omega)$
satisfying
\begin{align}
    \| T(t) f \|_{L^2_{d\mu}}
    \le
    \| f \|_{L^2_{d\mu}},\quad
    \| L T(t) f \|_{L^2_{d\mu}}
    \le
    \frac{1}{t} \| f \|_{L^2_{d\mu}}
\end{align}
for any
$f \in L^2_{d\mu}(\Omega)$.
\end{lemma}

We also recall the following
property proved in
\cite{SoWa16_JDE},
which will be used in Section 6:
\begin{lemma}[{\cite[Lemma 2.3]{SoWa16_JDE}}]\label{lem:D(L)}
We have
\begin{align}
    \left\{
    u \in H^2(\Omega) \cap H^1_0(\Omega) ;\,
    a(x)^{-1/2} \Delta u \in L^2(\Omega)
    \right\}
    \subset D(L).
\end{align}
\end{lemma}


For the inhomogeneous problem \eqref{eq:h:inhom},
applying \cite[Lemma 4.1.1, Proposition 4.1.6]{CaHa98},
we have the following well-posedness result.
\begin{theorem}\label{thm:h:wp}
Assume that
$v_0 \in D(L)$
and
$a(x)^{-1}G \in C^1([0,\infty); L^2_{d\mu}(\Omega))$.
Then, the function
$v$
defined by
\begin{align}
    v(t) = T(t)v_0 + \int_0^t T(t-s) (a(x)^{-1}G(s)) \,ds
\end{align}
is the unique solution to the problem
\eqref{eq:h:inhom} satisfying
\begin{align}
    v \in C([0,\infty); D(L)) \cap
    C^1([0,\infty); L^2_{d\mu}(\Omega)).
\end{align}
\end{theorem}

Next, we discuss the higher order
regularity in time for the solution of
\eqref{eq:h:inhom}.
We note that, by a formal straightforward computation,
the initial values of
$\partial_t^j v$
for
$j \ge 1$
are given by
\begin{align}\label{eq:dtv0}
    \partial_t^j v(x,0)
    &= [a(x)^{-1} \Delta]^j v_0(x)
    + \sum_{l=0}^{j-1} [a(x)^{-1}\Delta]^l a(x)^{-1} \partial_t^{j-1-l} G (x,0).
\end{align}

\begin{theorem}\label{thm:reg:h}
Let $k \ge 1$ be an integer,
$v_0 \in D(L)$,
and
$a(x)^{-1}G \in C^{k+1}([0,\infty); L^2_{d\mu}(\Omega))$.
Assume that
$\partial_t^j v(x,0)$
defined by the right-hand side of
\eqref{eq:dtv0}
satisfies
$\partial_t^j v(x,0) \in D(L)$
for $j=1,\ldots,k$.
Then, the solution $v$ to \eqref{eq:h:inhom}
obtained by Theorem \ref{thm:h:wp}
belongs to
\begin{align}\label{eq:v:reg}
    C^{k}([0,\infty); D(L)) \cap
    C^{k+1} ([0,\infty); L^2_{d\mu}(\Omega)).
\end{align}
\end{theorem}
\begin{proof}
When $k = 1$,
let $\psi = \psi(x,t)$
be the solution of \eqref{eq:h:inhom}
with the inhomogeneous term 
$\partial_t G$
and the initial data
$\psi(x,0) = a(x)^{-1}\Delta v_0(x) + a(x)^{-1}G(x,0)$.
Then, by Theorem \ref{thm:h:wp},
$\psi$ is given by
\begin{align}
    \psi(t)
    &=
    T(t) [a(x)^{-1}\Delta v_0(x) + a(x)^{-1}G(x,0)]
    + \int_0^t T(t-s) (a(x)^{-1}\partial_s G(s)) \,ds \\
    &=
    \partial_t T(t) v_0
    + T(t) [a(x)^{-1}G(x,0)] \\
    &\quad
    + \left[ T(t-s) (a(x)^{-1}G(s)) \right]_0^t
    + \int_0^t \partial_t T(t-s) (a(x)^{-1} G(s)) \,ds \\
    &=
    \partial_t T(t) v_0
    + a(x)^{-1}G(t) \\
    &\quad
    + \partial_t \int_0^t T(t-s) (a(x)^{-1}G(s)) \,ds
    - a(x)^{-1} G(t) \\
    &= \partial_t v(t).
\end{align}
Since
$\psi \in C([0,\infty);D(L)) \cap
C^1([0,\infty); L^2_{d\mu}(\Omega))$,
we have
\begin{align}
    v \in C^1([0,\infty); D(L))
    \cap C^2([0,\infty); L^2_{d\mu}(\Omega)),
\end{align}
that is, the assertion when $k=1$ is proved.
The general case
$k \ge 1$
can be proved in the same way with induction,
and we omit the detail.
\end{proof}

\subsection{A decomposition lemma}
In the following two subsections,
we give the idea of the
asymptotic expansion of
the solution of the damped wave equation
\eqref{dwx}.
To simplify the discussion,
we only give formal computation here.
The justification and the complete proof
of the asymptotic expansion will be 
given in Section 6.

Related to the initial-boundary value problem of
the damped wave equation \eqref{eq:gDW},
we consider the parabolic problem
with the same inhomogeneous term
$F$
and the initial data
$w_0(x) + a(x)^{-1}w_1(x)$:
\begin{align}\label{eq:gH}
     \left\{ \begin{array}{ll}
       a(x) \partial_t V - \Delta V = F, &x \in \Omega, t>0,\\
       V(x,t) = 0, & x \in \partial \Omega, t>0,\\
        V(x,0) = w_0(x) + a(x)^{-1} w_1(x),
        &x \in \Omega.
        \end{array} \right.
\end{align}
For the solution
$V$ of the above problem,
we further consider
the following initial-boundary value problem
of the damped wave equation
with the inhomogeneous term
$-\partial_t V$
and the initial data
$(U,\partial_t U)(x,0)
=(0,-a(x)^{-1}w_1(x))$:
\begin{align}
\label{eq:gU}
    \left\{ \begin{array}{ll}
        \partial_t^2 U - \Delta U + a(x) \partial_t U = - \partial_t V, &x \in \Omega, t>0,\\
        U(x,t) = 0, & x \in \partial \Omega, t>0,\\
        U(x,0) = 0, \ 
        \partial_t U(x,0) = - a(x)^{-1} w_1(x),
        &x \in \Omega.
        \end{array} \right.
\end{align}
Then, we have the following decomposition of
the solution $w$ to \eqref{eq:gDW}.
\begin{lemma}\label{lem_eq_U}
Let
$w$
be the solution of the damped wave equation \eqref{eq:gDW}
with the inhomogeneous term
$F$
and the initial data
$(w_0, w_1)$.
Let $V$ be the solution of the parabolic problem \eqref{eq:gH}, and
let $U$ be the solution of the
problem \eqref{eq:gU}.
Then, we have
\begin{align}
\label{u_v_U}
    w = V + \partial_t U.
\end{align}
\end{lemma}
\begin{proof}
Let
$\tilde{w} = V + \partial_t U$.
Then, we have
\begin{align*}
    \tilde{w}(x,0)
    = V(x,0) + \partial_t U(x,0)
    = w_0(x) + a(x)^{-1} w_1(x)
        - a(x)^{-1} w_1(x)
    = w_0(x).
\end{align*}
Also, by \eqref{eq:gU}, we obtain
\begin{align}
\label{pt_til_u}
    \partial_t \tilde{w}
    = \partial_t V + \partial_t^2 U
    = \Delta U - a(x) \partial_t U.
\end{align}
This implies
$\partial_t \tilde{w}(x,0)
= \Delta U(x,0) - a(x) \partial_t U(x,0)
= w_1(x)$.
Finally, differentiating
\eqref{pt_til_u} again,
and using the relation
$\partial_t^2 U = \partial_t \tilde{w} - \partial_t V$,
we deduce
\begin{align*}
    \partial_t^2 \tilde{w}
    &= \Delta \partial_t U - a(x) \partial_t^2 U \\
    &= \Delta (\tilde{w} - V)
        - a(x) ( \partial_t \tilde{w}  - \partial_t V) \\
    &= \Delta \tilde{w} - a(x) \partial_t \tilde{w} + F.
\end{align*}
Consequently,
$\tilde{w}$
is the solution of \eqref{eq:gDW}, and hence,
the uniqueness shows $\tilde{w} = w$.
This completes the proof.
\end{proof}

%
\subsection{Derivation of the asymptotic expansion}
Let
$u$
be the solution of \eqref{dwx}.
To expand $u$ in terms of solutions of
the corresponding parabolic problem,
we consider functions
$V_0, V_1, \ldots, V_n$
and the remainder terms
$U_1, U_2, \ldots, U_{n+1}$
successively defined in the following way:
first, we define
$V_0$ by
\begin{align}\label{eq:V_0}
     \left\{ \begin{array}{ll}
       a(x) \partial_t V_0 - \Delta V_0 = 0, &x \in \Omega, t>0,\\
       V_0(x,t) = 0, & x \in \partial \Omega, t>0,\\
        V_0(x,0) = u_0(x) + a(x)^{-1} u_1(x),
        &x \in \Omega.
        \end{array} \right.
\end{align}
and $U_1$ by
\begin{align}
\label{eq:U_1}
    \left\{ \begin{array}{ll}
        \partial_t^2 U_1 - \Delta U_1 + a(x) \partial_t U_1 = - \partial_t V_0, &x \in \Omega, t>0,\\
        U_1(x,t) = 0, & x \in \partial \Omega, t>0,\\
        U_1(x,0) = 0, \ 
        \partial_t U_1(x,0) = - a(x)^{-1} u_1(x),
        &x \in \Omega.
        \end{array} \right.
\end{align}
Then, by Lemma \ref{lem_eq_U}
we have the first decomposition
$u = V_0 + \partial_t U_1$.
According to the experiences, we expect that
this is the first-order asymptotic expansion of $u$
and therefore
$\partial_t U_1$ can be regarded as
a perturbation.
Next, to obtain the second-order expansion,
we further consider the decomposition of
$U_1$
in terms of corresponding parabolic problem.
Namely, we define $V_1$ by
\begin{align}\label{eq:V_1}
     \left\{ \begin{array}{ll}
       a(x) \partial_t V_1 - \Delta V_1 = - \partial_t V_0, &x \in \Omega, t>0,\\
       V_1(x,t) = 0, & x \in \partial \Omega, t>0,\\
        V_1(x,0) = - a(x)^{-2} u_1(x),
        &x \in \Omega.
        \end{array} \right.
\end{align}
and $U_2$ by
\begin{align}
\label{eq:U_2}
    \left\{ \begin{array}{ll}
        \partial_t^2 U_2 - \Delta U_2 + a(x) \partial_t U_2 = - \partial_t V_1, &x \in \Omega, t>0,\\
        U_2(x,t) = 0, & x \in \partial \Omega, t>0,\\
        U_2(x,0) = 0, \ 
        \partial_t U_2(x,0) = a(x)^{-2} u_1(x),
        &x \in \Omega.
        \end{array} \right.
\end{align}
Then, by Lemma \ref{lem_eq_U} again,
we have the decomposition
$U_1 = V_1 + \partial_t U_2$,
which implies
\begin{align}
    u = V_0 + \partial_t V_1 + \partial_t^2 U_2.
\end{align}
We expect that this gives the second-order
asymptotic expansion of $u$.
Repeating this procedure
for
$j = 2,3,\ldots, n$
we successively define
$V_j$ by
\begin{align}\label{eq:V_j}
     \left\{ \begin{array}{ll}
       a(x) \partial_t V_j - \Delta V_j = - \partial_t V_{j-1}, &x \in \Omega, t>0,\\
       V_j(x,t) = 0, & x \in \partial \Omega, t>0,\\
        V_j(x,0) = - (-a(x))^{-j-1} u_1(x),
        &x \in \Omega
        \end{array} \right.
\end{align}
and $U_{n+1}$ by
\begin{align}
\label{eq:U_n+1}
    \left\{ \begin{array}{ll}
        \partial_t^2 U_{n+1} - \Delta U_{n+1} + a(x) \partial_t U_{n+1} = - \partial_t V_n, &x \in \Omega, t>0,\\
        U_{n+1}(x,t) = 0, & x \in \partial \Omega, t>0,\\
        U_{n+1}(x,0) = 0, \ 
        \partial_t U_{n+1}(x,0) = (-a(x))^{-n-1} u_1(x),
        &x \in \Omega.
        \end{array} \right.
\end{align}
Then, we can have the
expected higher order decomposition
\begin{align}\label{eq:u_decomp}
    u = V_0 + \partial_t V_1
        + \partial_t^2 V_2 + \cdots
        + \partial_t^n V_n
        + \partial_t^{n+1} U_{n+1}.
\end{align}
By Theorems \ref{thm:WP:gDW} and \ref{thm:reg:gDW},
the existence, uniqueness, and regularity
of the solution
$U_{n+1}$
to \eqref{eq:U_n+1}
can be obtained
from the assumptions on the initial data
of Theorem \ref{thm:asym}.
The detail will be discussed in Section 6.

In the following sections,
we give energy estimates for
$V_0, V_1, \ldots, V_n$
and $U_{n+1}$
to justify that \eqref{eq:u_decomp}
actually gives the $n$-th order
asymptotic expansion of
$u$.

\section{
Energy estimates for the heat equation
}

We apply the weighted energy method
to obtain the decay estimate of
the parabolic problem
\begin{align}\label{eq:gH2}
     \left\{ \begin{array}{ll}
       a(x) \partial_t v - \Delta v = G, &x \in \Omega, t>0,\\
       v(x,t) = 0, & x \in \partial \Omega, t>0,\\
        v(x,0) = v_0(x),
        &x \in \Omega.
        \end{array} \right.
\end{align}

The goal of this section is
the following weighted energy estimates for
higher order derivatives of solutions to \eqref{eq:gH2}.
\begin{theorem}\label{thm:v}
Let
$k \ge 0$ be an integer,
$\delta \in (0,1/2)$,
$\varepsilon \in (0,1/2)$,
$\lambda \in [0,(1-2\delta)\gamma_{\ep})$
(see \eqref{gammatilde} for the definition of $\gamma_{\varepsilon}$),
$t_0 \ge 1$
and
$\beta = \lambda/(1-2\delta)$.
Let
$v_0 \in D(L)$,
$a(x)^{-1}G \in C^{k+1}([0,\infty); L^2_{d\mu}(\Omega))$,
and
let
$v$
be the corresponding solution of \eqref{eq:gH2}.
Moreover, we assume that
$\partial_t^j v(x,0)$
given by the right-hand side of
\eqref{eq:dtv0} satisfies
\begin{align}
    \partial_t^j v(x,0)
    &\in
    D(L) \cap 
    H^{0,(\lambda+2j)\frac{2-\alpha}{2}-\frac{\alpha}{2}}(\Omega),\\
    \nabla \partial_t^j v(x,0)
    &\in
    H^{0,(\lambda+1+2j)\frac{2-\alpha}{2}}(\Omega),
\end{align}
and
\begin{align}
    \int_{\Omega} a(x)^{-1} |\partial_t^j G(x,t)|^2 \Psi(x,t;t_0)^{\lambda+1+2j} \,dx
    \in L^1(0,\infty)
\end{align}
for $j=0,1,\ldots,k$.
Then, we have
\begin{align}
\label{eq:thm:v:conclusion:1}
    &\int_{\Omega} a(x) |\partial_t^j v(x,t)|^2 \Psi(x,t;t_0)^{\lambda+2j} \,dx
    \in L^{\infty}(0,\infty),\\
\label{eq:thm:v:conclusion:2}
    &\int_{\Omega} |\nabla \partial_t^j v(x,t)|^2 \Psi(x,t;t_0)^{\lambda+2j} \,dx
    \in L^{1}(0,\infty),\\
\label{eq:thm:v:conclusion:3}
    &\int_{\Omega} |\nabla \partial_t^j v(x,t)|^2 \Psi(x,t;t_0)^{\lambda+1+2j} \,dx
    \in L^{\infty}(0,\infty),\\
\label{eq:thm:v:conclusion:4}
    &\int_{\Omega} a(x) |\partial_t^{j+1} v(x,t)|^2 \Psi(x,t;t_0)^{\lambda+1+2j} \,dx
    \in L^1(0,\infty)
\end{align}
for $j=0,1,\ldots, k$.
\end{theorem}
We note that Theorem \ref{thm:reg:h}
ensures the regularity property
\eqref{eq:v:reg} for the solution $v$.
Thus, it suffices to show the
estimates
\eqref{eq:thm:v:conclusion:1}--\eqref{eq:thm:v:conclusion:4}.

The proof of Theorem \ref{thm:v} is based on
an induction argument.
The following lemma is the first step.
\begin{lemma}\label{lem:v1}
Under the assumptions on Theorem \ref{thm:v} with
$k=0$,
we have \eqref{eq:thm:v:conclusion:1}--\eqref{eq:thm:v:conclusion:4}
for $k=0$.
\end{lemma}
\begin{proof}
By Lemma \ref{lem.deltaphi},
Proposition \ref{prop:super-sol} (iii),
and the Schwarz inequality,
we calculate
\begin{align}
    \frac{d}{dt} \int_{\Omega} \frac{a(x) v^2}{\Phi_{\beta,\ep}^{1-2\delta}} \,dx
    &=
    2 \int_{\Omega} \frac{v \Delta v}{\Phi_{\beta,\ep}^{1-2\delta}} \,dx
    - (1-2\delta) \int_{\Omega}
        \frac{a(x) v^2 \partial_t \Phi_{\beta,\ep} }{\Phi_{\beta,\ep}^{2-2\delta}} \,dx \\
    &\quad +
        2\int_{\Omega} \frac{vG}{\Phi_{\beta,\ep}^{1-2\delta}} \,dx \\
    &\le - \frac{2\delta}{1-\delta}
        \int_{\Omega} \frac{|\nabla v|^2}{\Phi_{\beta,\ep}^{1-2\delta}} \,dx
            - (1-2\delta)
        \int_{\Omega}
            \frac{
                (a(x) \partial_t \Phi_{\beta,\ep} - \Delta \Phi_{\beta,\ep}) v^2
                }{\Phi_{\beta,\ep}^{2-2\delta}}
        \,dx \\
    &\quad +
        C \left( \int_{\Omega} a(x)v^2 \Psi^{\lambda-1} \,dx \right)^{1/2}
        \left( \int_{\Omega} a(x)^{-1} G^2
            \Psi^{\lambda+1} \,dx \right)^{1/2}.
\end{align}
From the Young inequality and
Proposition \ref{prop:super-sol} (iv),
we obtain
\begin{align}
    \frac{d}{dt} \int_{\Omega} \frac{a(x) v^2}{\Phi_{\beta,\ep}^{1-2\delta}} \,dx
    =
    - \frac{2\delta}{1-\delta}
        \int_{\Omega} \frac{|\nabla v|^2}{\Phi_{\beta,\ep}^{1-2\delta}} \,dx
        +
        C \int_{\Omega} a(x)^{-1} G^2
            \Psi^{\lambda+1} \,dx.
\end{align}
Integrating it over $[0,t]$
and using Proposition \ref{prop:super-sol} (iii), we deduce
\begin{align}
    &\int_{\Omega} a(x) v(x,t)^2 \Psi(x,t;t_0)^{\lambda} \,dx
    \in L^{\infty}(0,\infty),\\
\label{eq:vx:L1}
    &\int_{\Omega} |\nabla v(x,t)|^2 \Psi(x,t;t_0)^{\lambda} \,dx
    \in L^1(0,\infty),
\end{align}
Thus, we have \eqref{eq:thm:v:conclusion:1} and
\eqref{eq:thm:v:conclusion:2} in the case $j=0$.
We next compute
\begin{align}
\label{eq:nb_v_psi}
    \frac{d}{dt} \int_{\Omega} |\nabla v|^2 \Psi^{\lambda+1} \,dx
    &= (\lambda+1) \int_{\Omega} |\nabla v|^2 \Psi^{\lambda} \,dx
        + 2 \int_{\Omega} ( \nabla \partial_t v \cdot \nabla v ) \Psi^{\lambda+1} \,dx.
\end{align}
By \eqref{eq:vx:L1},
the first term of the right-hand side belongs to $L^1(0,\infty)$,
and by integration by parts, the second term is estimated as
\begin{align}
    &2 \int_{\Omega} ( \nabla \partial_t v \cdot \nabla v ) \Psi^{\lambda+1} \,dx \\
    &= -2 \int_{\Omega}
        (a(x) \partial_t v - G) \partial_t v \Psi^{\lambda+1} \,dx
        - (\lambda+1) \int_{\Omega} \partial_t v (\nabla v \cdot \nabla A_{\ep}(x))
                \Psi^{\lambda} \,dx \\
    &\le
        -2 \int_{\Omega} a(x) | \partial_t v|^2 \Psi^{\lambda+1} \,dx \\
    &\quad 
        + \frac{1}{2} \int_{\Omega} a(x) |\partial_t v|^2 \Psi^{\lambda+1} \,dx
        + 2 \int_{\Omega} a(x)^{-1} G^2 \Psi^{\lambda+1} \,dx \\
    &\quad 
        + \frac{1}{2} \int_{\Omega} a(x) | \partial_t v|^2 A_{\ep}(x) \Psi^{\lambda} \,dx
        + 2(\lambda+1)^2 \int_{\Omega} \frac{|\nabla A_{\ep}(x)|^2}{a(x)A_{\ep}(x)} |\nabla v|^2
                    \Psi^{\lambda} \,dx \\
    &\le - \int_{\Omega} a(x) | \partial_t v|^2 \Psi^{\lambda+1} \,dx \\
    &\quad
        +C \int_{\Omega} |\nabla v|^2 \Psi^{\lambda} \,dx
        + 2 \int_{\Omega} a(x)^{-1} G^2 \Psi^{\lambda+1} \,dx.
\label{eq:lem:v1:en1}
\end{align}
Here, we have used the property \eqref{A3} in Lemma \ref{lem_A_ep} and
the relation $A_{\ep}(x) \le \Psi(x,t;t_0)$,
which follows from the definition
$\Psi$
(see \eqref{psi}).
By \eqref{eq:vx:L1} and the assumption on $G$,
the last two terms of the right-hand side of above are in $L^1(0,\infty)$.
Thus, integrating \eqref{eq:nb_v_psi} over $[0,t]$, we conclude
\begin{align*}
    &\int_{\Omega} |\nabla v(x,t)|^2 \Psi(x,t;t_0)^{\lambda+1} \,dx
    \in L^{\infty}(0,\infty),\\
    &\int_{\Omega} a(x) | \partial_t v(x,t)|^2 \Psi(x,t;t_0)^{\lambda+1} \,dx
    \in L^1(0,\infty),
\end{align*}
that is, \eqref{eq:thm:v:conclusion:3} and \eqref{eq:thm:v:conclusion:4}
in the case $j=0$,
and the proof is now complete.
\end{proof}

\begin{remark}
It should be noted that
the integration by parts in the above proof
can be justified
completely
by the approximation
argument in the same as
\cite[Section 4]{SoWa21_JMSJ}.
\end{remark}

Next, we prove the following lemma,
which is the main part of the induction argument of
the proof of Theorem \ref{thm:v}.

\begin{lemma}\label{lem:v2}
Assume $a(x)$ satisfies \eqref{a}.
Let 
$\sigma \ge 0$,
$t_0 \ge 1$,
$v_0 \in D(L)$
and
$a(x)^{-1}G \in C^1([0,\infty); L^2_{d\mu}(\Omega))$.
Let
$v$
be the corresponding solution of \eqref{eq:gH2}.
We further assume
\begin{align}
    &v_0 \in
    H^{0,\frac{2-\alpha}{2}\sigma-\frac{\alpha}{2}}(\Omega),
    \quad
    \nabla v_0 \in
    H^{0,\frac{2-\alpha}{2}(\sigma + 1)}(\Omega),\\
\label{eq:lem:v2:indass2}
    &\int_{\Omega} a(x)^{-1} |G(x,t)|^2 \Psi(x,t;t_0)^{\sigma+1} \,dx
    \in L^1(0,\infty),
\end{align}
and also
\begin{align}
\label{eq:lem:v2:indass1}
    &\int_{\Omega} a(x) |v(x,t)|^2 \Psi(x,t;t_0)^{\sigma-1}\,dx
    \in L^1(0,\infty).
\end{align}
Then, we have
\begin{align}
\label{eq:lem:v2:conc2}
    &\int_{\Omega} a(x) |v(x,t)|^2
    \Psi(x,t;t_0)^{\sigma} \,dx
    \in L^{\infty} (0,\infty),\\
\label{eq:lem:v2:conc3}
    &\int_{\Omega} |\nabla v(x,t)|^2
    \Psi(x,t;t_0)^{\sigma} \,dx
    \in L^{1} (0,\infty),\\
\label{eq:lem:v2:conc4}
    &\int_{\Omega} |\nabla v(x,t)|^2
    \Psi(x,t;t_0)^{\sigma + 1} \,dx
    \in L^{\infty} (0,\infty),\\
\label{eq:lem:v2:conc1}
    &\int_{\Omega} a(x) |\partial_t v(x,t)|^2
        \Psi(x,t;t_0)^{\sigma+1} \,dx
    \in L^1(0,\infty).
\end{align}
\end{lemma}
\begin{remark}
In the proof of Theorem \ref{thm:v},
we will choose
$\sigma = \lambda + 2j$
for $j \in \mathbb{N}$.
\end{remark}
\begin{proof}[Proof of Lemma \ref{lem:v2}]
Suppose \eqref{eq:lem:v2:indass2}
and \eqref{eq:lem:v2:indass1}.
Similarly as the proof of Lemma \ref{lem:v1},
we compute
\begin{align}
    &\frac{d}{dt} \int_{\Omega} a(x) | v(x,t)|^2 \Psi(x,t;t_0)^{\sigma} \,dx \\
    &= 2 \int_{\Omega} v (\Delta v) \Psi^{\sigma} \,dx
     + \sigma \int_{\Omega} a(x) | v|^2 \Psi^{\sigma -1} \,dx \\
    &\quad + 2 \int_{\Omega} v G \Psi^{\sigma} \,dx
\label{eq:lem:v2:en1}
\end{align}
The second term of the right-hand side is in
$L^1(0,\infty)$
due to the assumption \eqref{eq:lem:v2:indass1}.
Noting the relation
$2v \Delta v = - 2|\nabla v|^2 + \Delta(|v|^2)$
and using the integration by parts,
we calculate the first term of the right-hand side as
\begin{align}
    2 \int_{\Omega} v (\Delta v) \Psi^{\sigma} \,dx
    &= - 2 \int_{\Omega} |\nabla v|^2 \Psi^{\sigma}\,dx
        + \int_{\Omega} | v|^2 \Delta (\Psi^{\sigma}) \,dx.
\end{align}
The last term of the above is
further estimated by
\begin{align}
    \int_{\Omega} |v|^2
    \left| \Delta (\Psi^{\sigma}) \right| \,dx
    &= \sigma \int_{\Omega} |v|^2
    \left|
        \Delta A_{\varepsilon} \Psi
        + (\sigma-1) |\nabla A_{\varepsilon}|^2
    \right|
    \Psi^{\sigma-2} \,dx \\
    &\le C \int_{\Omega} a(x) |v|^2
        \Psi^{\sigma-1} \,dx,
\end{align}
where we have used
\eqref{A1}, \eqref{A3}, and
$A_{\varepsilon} \le \Psi$.
Therefore, this term also belongs to
$L^1(0,\infty)$
by the assumption \eqref{eq:lem:v2:indass1}.
Finally, we apply the Schwarz inequality
to the last term of \eqref{eq:lem:v2:en1}
and obtain
\begin{align}
    2\int_{\Omega} v G \Psi^{\sigma} \,dx
    &\le \int_{\Omega} a(x) |v|^2
        \Psi^{\sigma-1} \,dx
        + \int_{\Omega} a(x)^{-1} |G|^2 \Psi^{\sigma+1} \,dx
\end{align}
and these are in $L^1(0,\infty)$
due to the assumptions \eqref{eq:lem:v2:indass2}
and \eqref{eq:lem:v2:indass1}.
Consequently, integrating \eqref{eq:lem:v2:en1}
over $[0,t]$,
we have
\begin{align}
    &\int_{\Omega} a(x) |v(x,t)|^2
    \Psi(x,t;t_0)^{\sigma} \,dx
    \in L^{\infty} (0,\infty),\\
    &\int_{\Omega} |\nabla v(x,t)|^2 \Psi(x,t;t_0)^{\sigma} \,dx
    \in L^1(0,\infty).
\end{align}
Thus, we have \eqref{eq:lem:v2:conc2} and \eqref{eq:lem:v2:conc3}.

Next, we compute
\begin{align}
    \frac{d}{dt} \int_{\Omega} |\nabla v(x,t)|^2 \Psi(x,t;t_0)^{\sigma+1} \,dx
    &=(\sigma+1)
    \int_{\Omega} |\nabla v|^2 \Psi^{\sigma} \,dx \\
    &\quad + 
    2 \int_{\Omega} (\nabla \partial_t v \cdot \nabla v) \Psi^{\sigma + 1} \,dx.
\label{eq:lem:v2:en2}
\end{align}
The first term of the right-hand side
belongs to $L^1(0,\infty)$ by \eqref{eq:lem:v2:conc3}.
The second term can be estimated
in completely the same way as \eqref{eq:lem:v1:en1},
and we have
\begin{align}
    2 \int_{\Omega} (\nabla \partial_t v \cdot \nabla v) \Psi^{\sigma + 1} \,dx
    &\le - \int_{\Omega} a(x) |\partial_t v|^2 \Psi^{\sigma+1} \,dx \\
    &\quad + C \int_{\Omega} |\nabla v|^2 \Psi^{\sigma} \,dx \\
    &\quad + C \int_{\Omega} a(x)^{-1} |G|^2 \Psi^{\sigma+1} \,dx.
\end{align}
By using
\eqref{eq:lem:v2:indass2}, and \eqref{eq:lem:v2:conc3},
the last two terms of above
belong to $L^1(0,\infty)$.
Finally, integrating \eqref{eq:lem:v2:en2} over $[0,t]$, we conclude 
\begin{align}
    &\int_{\Omega} |\nabla v(x,t)|^2
    \Psi(x,t;t_0)^{\sigma+1} \,dx
    \in L^{\infty} (0,\infty),\\
    &\int_{\Omega} a(x) |\partial_t v(x,t)|^2
    \Psi(x,t;t_0)^{\sigma+1} \,dx
    \in L^1(0,\infty).
\end{align}
This completes the proof of \eqref{eq:lem:v2:conc1} and \eqref{eq:lem:v2:conc4}.
\end{proof}
\begin{proof}[Proof of Theorem \ref{thm:v}]
We note that the case $k=0$ has been already proved by Lemma \ref{lem:v1}.
Let $k \ge 1$ be an integer.
Then, by Lemma \ref{lem:v1}, we have
\eqref{eq:thm:v:conclusion:1}--\eqref{eq:thm:v:conclusion:4}
in the case $j=0$.

Next,
for $j=1$,
we apply Lemma \ref{lem:v2} with
$\sigma = \lambda + 2j$
and with the replacement of
$v$ and $G$ by $\partial_t v$ and $\partial_t G$, respectively.
We remark that the condition \eqref{eq:lem:v2:indass1} with
$\sigma = \lambda + 2j$
is fulfilled
by virtue of \eqref{eq:thm:v:conclusion:4} with $j=0$.
Then, we obtain
\eqref{eq:lem:v2:conc2}--\eqref{eq:lem:v2:conc1} for
$\sigma = \lambda + 2j$
with the replacement of
$v$ by $\partial_t v$,
namely, we reach the conclusions
\eqref{eq:thm:v:conclusion:1}--\eqref{eq:thm:v:conclusion:4} for $j=1$.

The properties
\eqref{eq:thm:v:conclusion:1}--\eqref{eq:thm:v:conclusion:4} for $j=1$
allow us to apply again Lemma \ref{lem:v2} with
$\sigma = \lambda + 2j$,
$j=2$
and with the replacement of
$v$ and $G$ by $\partial_t^2 v$ and $\partial_t^2 G$, respectively.
Then, we can see that
\eqref{eq:thm:v:conclusion:1}--\eqref{eq:thm:v:conclusion:4} for $j=2$ hold.
Repeating this argument until $j=k$,
we complete the proof of Theorem \ref{thm:v}.
\end{proof}

\section{Energy estimates for the damped wave equation}
\subsection{First order energy estimates}
In this section, we discuss the energy estimate
for the general damped wave equation \eqref{eq:gDW}.

The results of this section
will be used in the next section
by putting
$w=U_{n+1}$, $F=-\partial_t V_n$,
$w_0 = 0$, and $w_1 = (-a(x))^{-n-1}u_1(x)$
(see \eqref{eq:U_n+1})
to derive the energy estimate of
$\partial_t^{n+1}U_{n+1}$.

We start with the definition of the weighted energy of $w$.
\begin{definition}
For
$\delta \in(0,1/2)$, $\varepsilon \in (0,1/2)$, $\lambda \in [0, (1-2\delta)\gamma_{\varepsilon})$
(see \eqref{gammatilde} for the definition of $\gamma_{\varepsilon}$),
$\beta = \lambda/(1-2\delta)$, $t_0 \ge 1$, and $\nu >0$,
we define
\begin{align}
\label{eq:en:w:1}
    E_1[w](t;t_0,\lambda) &:=
    \int_{\Omega} \left( |\nabla w(x,t)|^2 + |\partial_t w(x,t)|^2 \right) \Psi(x,t;t_0)^{\lambda+1} \,dx,\\
\label{eq:en:w:0}
    E_0[w](t;t_0,\lambda) &:=
    \int_{\Omega} \left( 2 w(x,t) \partial_tw(x,t) + a(x) |w(x,t)|^2 \right) \Phi_{\beta,\varepsilon}(x,t;t_0)^{-1+2\delta} \,dx,\\
\label{eq:en:w:1+0}
    E[w](t;t_0,\lambda,\nu) &:=
    \nu E_1[w](t;t_0,\lambda) + E_0[w](t;t_0,\lambda)
\end{align}
for $t \ge 0$.
\end{definition}
We note that, for any
$\nu > 0$,
there exists
$t_1 >0$
such that
\begin{align}
    E[w](t;t_0,\lambda,\nu)
    \sim E_1[w](t;t_0,\lambda)
        + \int_{\Omega} a(x)|w(x,t)|^2 \Psi(x,t;t_0)^{\lambda} \,dx
\end{align}
holds for any $t_0 \ge t_1$.
Indeed, the Schwarz inequality implies
\begin{align}
    |2 w \partial_t w|
    \le \frac{a(x)}{2} |w|^2 + 2 a(x)^{-1} |\partial_t w|^2,
\end{align}
and Proposition \ref{prop:super-sol} (iii) and (iv)
lead to 
\begin{align}
    2a(x)^{-1}|\partial_t w|^2 \Phi_{\beta,\varepsilon}^{-1+2 \delta}
    \le C \Psi^{\frac{\alpha}{2-\alpha}}
    |\partial_t w|^2 \Psi^{\lambda}
    \le C t_0^{-1+\frac{\alpha}{2-\alpha}}
    |\partial_t w|^2 \Psi^{\lambda+1}
    \le \frac{\nu}{2} |\partial_t w|^2 \Psi^{\lambda+1}
\end{align}
for sufficiently large $t_0$.

The main theorem of this subsection is the following:
\begin{theorem}\label{thm:en:w:1}
Assume that $a(x)$ satisfies \eqref{a}.
Let
$\delta \in(0,1/2)$,
$\varepsilon \in (0,1/2)$,
$\lambda \in [0, (1-2\delta)\gamma_{\varepsilon})$,
and
$\beta = \lambda/(1-2\delta)$.
Then, there exist constants
$\nu=\nu(N,\alpha,\delta,\varepsilon,\lambda)>0$
and
$t_*=t_*(N,\alpha,\delta,\varepsilon,\lambda,\nu) \ge 1$
such that
for any $t_0 \ge t_*$, the following holds:
Let
$m=(\lambda+1)\frac{2-\alpha}{2}$
and assume
$F \in C^1([0,\infty); H^{0,m}(\Omega))$
satisfies
\begin{align}
    \int_{\Omega} a(x)^{-1} |F(x,t)|^2 \Psi(x,t;t_0)^{\lambda+1} \,dx \in L^1(0,\infty).
\end{align}
Let $w$ be the solution of \eqref{eq:gDW} with the initial data
$(w_0, w_1) \in (H^{2,m}(\Omega)\cap H_0^{1,m}(\Omega) ) \times H^{1,m}_0(\Omega)$
given in Theorem \ref{thm:WP:gDW}.
Then, we have
\begin{align}
    &E[w](t;t_0,\lambda,\nu) \in L^{\infty}(0,\infty),\\
    &\int_{\Omega} |\nabla w(x,t)|^2 \Psi(x,t;t_0)^{\lambda} \,dx \in L^1(0,\infty),\\
    &\int_{\Omega} a(x) |\partial_t w(x,t)|^2 \Psi(x,t;t_0)^{\lambda+1} \,dx \in L^1(0,\infty).
\end{align}
\end{theorem}

The proof of Theorem \ref{thm:en:w:1}
is a bit lengthy, but
the outline is as follows.
We shall derive good terms
$\displaystyle - \int_{\Omega} |\nabla w|^2 \Psi^{\lambda}\,dx$
and
$\displaystyle -\int_{\Omega} a(x)| \partial_tw|^2 \Psi^{\lambda+1} \,dx$
from the computations of
$\frac{d}{dt}E_0[w](t;t_0,\lambda)$
and
$\frac{d}{dt} E_1[w](t;t_0,\lambda)$,
respectively
(see the right-hand sides of Lemmas \ref{lem:en:w:E0} and \ref{lem:en:w:E1}).
Then, we sum up them with sufficiently small
$\nu$
and sufficiently large $t_0$
so that the other bad terms
are absorbed by these good terms.

We first give estimates of
$E_0[w](t;t_0,\lambda)$.

\begin{lemma}\label{lem:en:w:E0}
Under the assumptions on Theorem \ref{thm:en:w:1},
for any
$t_0 \ge 1$ and $t>0$,
we have
\begin{align}\label{eq:lem:en:w:E0:1}
    \frac{d}{dt} E_0[w](t;t_0,\lambda) 
    &\le
    - \eta_0 \int_{\Omega} |\nabla w(x,t)|^2 \Psi(x,t;t_0)^{\lambda} \,dx \\
    &\quad
    - \eta_0 \int_{\Omega} a(x) |w(x,t)|^2 \Psi(x,t;t_0)^{\lambda-1} \,dx \\
    &\quad 
        + C \int_{\Omega} |\partial_t w(x,t)|^2 \Psi(x,t;t_0)^{\lambda} \,dx \\
    &\quad
        + C \int_{\Omega} a(x)^{-1} |F(x,t)|^2 \Psi(x,t;t_0)^{\lambda+1} \,dx
\end{align}
with some constants
$\eta_0 = \eta_0(\varepsilon,\delta) > 0$
and 
$C=C(N,\alpha,\delta,\varepsilon,\lambda)>0$.
\end{lemma}
\begin{proof}
By the definition of $E_0[w](t;t_0,\lambda)$
and using the equation \eqref{eq:gDW},
we calculate
\begin{align}
    \frac{d}{dt}E_0[w](t;t_0,\lambda)
    &= 2 \int_{\Omega} |\partial_tw|^2 \Phi_{\beta,\varepsilon}^{-1+2\delta} \,dx
        + 2 \int_{\Omega} w \left( \partial_t^2 w + a(x) \partial_t w \right) \Phi_{\beta,\varepsilon}^{-1+2\delta} \,dx \\
    &\quad - (1-2\delta) \int_{\Omega} \left( 2 w \partial_t w + a(x) |w|^2 \right)
               ( \partial_t \Phi_{\beta,\varepsilon} ) \Phi_{\beta,\varepsilon}^{-2+2\delta} \,dx \\
    &= 2 \int_{\Omega} |\partial_tw|^2 \Phi_{\beta,\varepsilon}^{-1+2\delta} \,dx
        + 2 \int_{\Omega} w \left( \Delta w + F \right) \Phi_{\beta,\varepsilon}^{-1+2\delta} \,dx \\
        &\quad - (1-2\delta) \int_{\Omega} \left( 2 w \partial_t w + a(x) |w|^2 \right)
               ( \partial_t \Phi_{\beta,\varepsilon} ) \Phi_{\beta,\varepsilon}^{-2+2\delta} \,dx.
\end{align}
Applying Lemma \ref{lem.deltaphi}, we have
\begin{align}
    \frac{d}{dt}E_0[w](t;t_0,\lambda)
    &\le 2 \int_{\Omega} |\partial_tw|^2 \Phi_{\beta,\varepsilon}^{-1+2\delta} \,dx
        - \frac{2\delta}{1-\delta} \int_{\Omega} |\nabla w|^2 \Phi_{\beta,\varepsilon}^{-1+2\delta} \,dx \\
    &\quad - (1-2\delta) \int_{\Omega} |w|^2
                \left( a(x) \partial_t \Phi_{\beta,\varepsilon} - \Delta \Phi_{\beta,\varepsilon} \right)
                \Phi_{\beta,\varepsilon}^{-2+2\delta} \,dx \\
    &\quad - 2 (1-2\delta) \int_{\Omega} w \partial_t w  ( \partial_t \Phi_{\beta,\varepsilon} ) \Phi_{\beta,\varepsilon}^{-2+2\delta} \,dx \\
    &\quad + 2 \int_{\Omega} w F \Phi_{\beta,\varepsilon}^{-1+2\delta} \,dx.
\label{eq:lem:en:w:E0:2}
\end{align}
By Proposition \ref{prop:super-sol} (iii) and (iv),
the third term of the right-hand side of \eqref{eq:lem:en:w:E0:2}
is estimated as
\begin{align}\label{eq:lem:en:w:E0:wL2}
    - (1-2\delta) \int_{\Omega} |w|^2
                \left( a(x) \partial_t \Phi_{\beta,\varepsilon} - \Delta \Phi_{\beta,\varepsilon} \right)
                \Phi_{\beta,\varepsilon}^{-2+2\delta} \,dx 
    \le 
    - \eta \int_{\Omega} a(x) |w|^2 \Psi^{\lambda -1} \,dx
\end{align}
with some $\eta > 0$.
Moreover, by Proposition \ref{prop:super-sol} (iii),
the second term of the right-hand side of \eqref{eq:lem:en:w:E0:2} is estimated as
\begin{align}
\label{eq:lem:en:w:E0:5}
    - \frac{2\delta}{1-\delta} \int_{\Omega} |\nabla w|^2 \Phi_{\beta,\varepsilon}^{-1+2\delta} \,dx
    \le - \eta' \int_{\Omega} |\nabla w|^2 \Psi^{\lambda} \,dx
\end{align}
with some $\eta' > 0$.
Moreover, we can drop the third term of the right-hand side,
and we also have
$|\partial_t \Phi_{\beta,\varepsilon}| = |\beta \Phi_{\beta+1,\varepsilon} | \le C \Psi^{-\beta -1}$,
which implies
\begin{align}
    \left| w \partial_t w  ( \partial_t \Phi_{\beta,\varepsilon} ) \Phi_{\beta,\varepsilon}^{-2+2\delta} \right|
    \le C |w| |\partial_t w| \Psi^{\lambda-1}.
\end{align}
Therefore, from the above inequality
with
the Schwarz inequality, we estimate
the fourth term of the right-hand side of \eqref{eq:lem:en:w:E0:2} as
\begin{align}
   &\left| \int_{\Omega} w \partial_t w  ( \partial_t \Phi_{\beta,\varepsilon} ) \Phi_{\beta,\varepsilon}^{-2+2\delta} \,dx \right| \\
   &\le C \int_{\Omega}
        |w| |\partial_t w| \Psi^{\lambda-1} \,dx \\
   &\le C \left( \int_{\Omega} a(x) |w|^2 \Psi^{\lambda-1} \,dx \right)^{1/2}
        \left( \int_{\Omega} a(x)^{-1} |\partial_t w|^2 \Psi^{\lambda-1} \,dx \right)^{1/2} \\
    &\le
    \eta_1
    \int_{\Omega} a(x) |w|^2 \Psi^{\lambda-1} \,dx
   + C(\eta_1) \int_{\Omega} |\partial_t w|^2 \Psi^{\lambda} \,dx
\label{eq:lem:en:w:E0:3}
\end{align}
for any $\eta_1>0$.
Similarly, the last term of the right-hand side of \eqref{eq:lem:en:w:E0:2} is estimated as
\begin{align}
    \left| 2\int_{\Omega} w F \Phi_{\beta,\varepsilon}^{-1+2\delta} \,dx \right|
    &\le C \int_{\Omega} |w| |F| \Psi^{\lambda} \,dx \\
    &\le \left( \int_{\Omega} a(x) |w|^2 \Psi^{\lambda-1} \,dx \right)^{1/2}
        \left( \int_{\Omega} a(x)^{-1} |F|^2 \Psi^{\lambda+1} \,dx \right)^{1/2} \\
    &\le \eta_2 \int_{\Omega} a(x) |w|^2 \Psi^{\lambda-1} \,dx
        + C \int_{\Omega} a(x)^{-1} |F|^2 \Psi^{\lambda+1} \,dx
\label{eq:lem:en:w:E0:4}
\end{align}
for any $\eta_2>0$.
Therefore, by applying
\eqref{eq:lem:en:w:E0:5}--\eqref{eq:lem:en:w:E0:4} to
\eqref{eq:lem:en:w:E0:2},
we have the desired estimate.
\end{proof}

\begin{lemma}\label{lem:en:w:E1}
Under the assumptions on Theorem \ref{thm:en:w:1},
there exists $t_2 \ge 1$ such that
for any $t_0 \ge t_2$ and $t>0$,
we have
\begin{align}
    \frac{d}{dt} E_1[w](t;t_0,\lambda)
    &\le - \int_{\Omega} a(x) |\partial_t w(x,t)|^2 \Psi(x,t;t_0)^{\lambda+1} \,dx
        + C \int_{\Omega} |\nabla w(x,t)|^2 \Psi(x,t;t_0)^{\lambda} \,dx \\
    &\quad + C \int_{\Omega} a(x)^{-1} |F(x,t)|^2 \Psi(x,t;t_0)^{\lambda+1} \,dx
\end{align}
with some constant
$C = C(N,\alpha,\delta,\varepsilon,\lambda,t_2) >0$.
\end{lemma}
\begin{proof}
By the definition of $E_1[w](t;t_0,\lambda)$
and the equation \eqref{eq:gDW},
we calculate
\begin{align}
    \frac{d}{dt} E_1[w](t;t_0,\lambda)
    &= 2 \int_{\Omega} \left( \nabla \partial_t w \cdot \nabla w + \partial_t w \partial_t^2 w \right) \Psi^{\lambda+1} \,dx \\
    &\quad
        + (\lambda+1) \int_{\Omega} \left( |\nabla w|^2 + |\partial_t w|^2 \right) \Psi^{\lambda} \,dx \\
    &= 2 \int_{\Omega} \partial_t w \left( -\Delta w + \partial_t^2 w \right) \Psi^{\lambda+1} \,dx \\
    &\quad
        -2 (\lambda+1) \int_{\Omega} \partial_t w ( \nabla w \cdot \nabla \Psi ) \Psi^{\lambda} \,dx \\
    &\quad
        + (\lambda+1) \int_{\Omega} \left( |\nabla w|^2 + |\partial_t w|^2 \right) \Psi^{\lambda} \,dx \\
    &= -2 \int_{\Omega} a(x) |\partial_t w|^2 \Psi^{\lambda+1} \,dx + 2 \int_{\Omega} \partial_t w F \Psi^{\lambda+1} \,dx \\
    &\quad 
       -2 (\lambda+1) \int_{\Omega} \partial_t w ( \nabla w \cdot \nabla \Psi ) \Psi^{\lambda} \,dx \\
    &\quad
        + (\lambda+1) \int_{\Omega} \left( |\nabla w|^2 + |\partial_t w|^2 \right) \Psi^{\lambda} \,dx.
\label{eq:lem:en:w:E1:2}
\end{align}
Here, we note that the integration by parts
in the second identity
is justified, since
$\partial_t w \nabla w \Psi^{\lambda+1} \in L^1(\Omega)$
for each $t \ge 0$.
For the second term of the right-hand side of \eqref{eq:lem:en:w:E1:2}, we apply the Schwarz inequality to obtain
\begin{align}
\label{eq:lem:en:w:E1:3}
    | 2 \partial_t w F | \le \frac{a(x)}{4} |\partial_t w|^2 + 4 a(x)^{-1} |F|^2.
\end{align}
Next, by the Schwarz inequality,
the third term of the right-hand side of \eqref{eq:lem:en:w:E1:2} is estimated as
\begin{align}
    | -2 (\lambda+1) \partial_t w (\nabla w \cdot \nabla \Psi) |
    &\le \frac{a(x)}{4} |\partial_t w|^2 \Psi + C |\nabla w|^2 \frac{|\nabla \Psi|^2}{a(x)\Psi} \\
    &\le \frac{a(x)}{4} |\partial_t w|^2 \Psi + C |\nabla w|^2,
\label{eq:lem:en:w:E1:4}
\end{align}
where we have also used
\begin{align}
    \frac{|\nabla \Psi|^2}{a(x) \Psi} \le \frac{|\nabla A_{\varepsilon} (x)|^2}{a(x) A_{\varepsilon} (x)} \le \frac{2-\alpha}{N-\alpha} + \varepsilon,
\end{align}
which follows from \eqref{A3}.
Moreover, for the last term of the right-hand side of \eqref{eq:lem:en:w:E1:2},
we note that
$\Psi^{-1} \le t_0^{-1+\frac{\alpha}{2-\alpha}} A_{\varepsilon}(x)^{-\frac{\alpha}{2-\alpha}}
\le C t_0^{-1+\frac{\alpha}{2-\alpha}} a(x) \le \frac{a(x)}{2(\lambda+1)}$
holds for $t_0 \ge t_2$, provided that $t_2$ is sufficiently large.
Thus, we have
\begin{align}
\label{eq:lem:en:w:E1:5}
    \left| (\lambda+1) \int_{\Omega} |\partial_t w|^2 \Psi^{\lambda} \,dx \right|
    &\le \frac{1}{2} \int_{\Omega} a(x) |\partial_t w|^2 \Psi^{\lambda+1} \,dx.
\end{align}
Finally, applying \eqref{eq:lem:en:w:E1:3}--\eqref{eq:lem:en:w:E1:5} to \eqref{eq:lem:en:w:E1:2},
we have the desired estimate.
\end{proof}

Now we are in the position to prove Theorem \ref{thm:en:w:1}.
\begin{proof}[Proof of Theorem \ref{thm:en:w:1}]
Let
$t_2$
be the constant given in Lemma \ref{lem:en:w:E1}.
For $t_0 \ge t_2$,
by Lemmas \ref{lem:en:w:E0} and \ref{lem:en:w:E1},
we calculate
\begin{align}
    \frac{d}{dt} E[w](t;t_0,\lambda,\nu)
    &\le \int_{\Omega} \left( - \nu a(x)  + C \Psi^{-1} \right) |\partial_t w|^2 \Psi^{\lambda+1} \,dx \\
    &\quad
        + \left( \nu C - \eta_0 \right) \int_{\Omega} |\nabla w|^2 \Psi^{\lambda} \,dx \\
    &\quad 
        + C \int_{\Omega} a(x)^{-1} |F|^2 \Psi^{\lambda+1} \,dx.
\label{eq:thm:en:w:1:2}
\end{align}
By taking $\nu > 0$ sufficiently small so that
$\nu C - \eta_0 < 0$.
After that, taking
$t_* \ge t_2$ sufficiently large so that
\begin{align}
    - \nu a(x)  + C \Psi^{-1}
    &\le - \nu a(x) + C t_0^{-1+\frac{\alpha}{2-\alpha}} a(x)
    \le - \frac{\nu}{2} a(x)
\end{align}
holds for $t_0 \ge t_*$.
Therefore, integrating \eqref{eq:thm:en:w:1:2} over $[0,t]$, we have
\begin{align}
   &E[w](t;t_0,\lambda,\nu)
    + \eta_* \int_0^t \int_{\Omega} a(x) |\partial_t w|^2 \Psi^{\lambda+1} \,dx d\tau \\
    &\quad
    + \eta_* \int_0^t \int_{\Omega} |\nabla w|^2 \Psi^{\lambda} \,dx d\tau \\
    &\le E[w](0;t_0,\lambda,\nu)
        + C \int_{0}^{t} \int_{\Omega} a(x)^{-1} |F|^2 \Psi^{\lambda+1} \,dx d\tau
\end{align}
with some constant
$\eta_* > 0$.
By the assumptions of the theorem, the right-hand side is bounded with respect to $t>0$.
Thus, we complete the proof.
\end{proof}

\subsection{Higher order energy estimates}
\begin{definition}
Let $\delta\in (0,1/2)$,
$\varepsilon\in(0,1)$,
$\lambda \in [0,(1-2\delta)\gamma_{\varepsilon})$,
where $\gamma_{\varepsilon}$
is defined in \eqref{gammatilde},
and let $\nu >0$.
For an integer $j \ge 1$, we define
\begin{align}
     E_1^{(j)} [w](t;t_0,\lambda)
    &:= \int_{\Omega} \left( |\nabla w(x,t)|^2 + |\partial_t w(x,t)|^2 \right) \Psi(x,t;t_0)^{\lambda+1+2j} \,dx,\\
    E_0^{(j)} [w](t;t_0,\lambda)
    &:= \int_{\Omega} \left( 2w(x,t) \partial_t w(x,t) + a(x) |w(x,t)|^2 \right) \Psi(x,t;t_0)^{\lambda+2j} \,dx,\\
    E^{(j)}[w](t;t_0,\lambda,\nu)
    &:= \nu E_1^{(j)}[w](t;t_0,\lambda) + E_0^{(j)}[w](t;t_0,\lambda)
\end{align}
for $t \ge 0$.
\end{definition}
We remark that there exists $t_1 > 0$
such that for any $t_0 \ge t_1$ and $t \ge 0$,
we have
\begin{align}
    E^{(j)} [w] (t;t_0,\lambda,\nu)
    \sim
    \int_{\Omega}
    \left[ \left( |\nabla w|^2 + |\partial_t w|^2 \right)\Psi^{\lambda+1+2j}
    + a(x)|w|^2 \Psi^{\lambda+2j} \right] \,dx.
\end{align}

The main result of this subsection
is the following energy estimates
for higher order derivatives of
the solution of the damped wave equation \eqref{eq:gDW}.

\begin{theorem}\label{thm:en:w:k}
Let $k \ge 1$ be an integer.
Assume that $a(x)$ satisfies \eqref{a}.
Let $\delta \in (0,1/2)$,
$\varepsilon\in (0,1/2)$,
and
$\lambda \in [0,(1-2\delta)\gamma_{\varepsilon})$.
Then, there exist constants
$\nu^{(j)} = \nu^{(j)}(N,\alpha,\lambda,j) >0$
and
$t_*^{(j)} = t_*^{(j)}(N,\alpha,\delta,\varepsilon,\lambda,\nu^{(j)}, j) \ge 1$
for $j=1,\ldots,k$
such that for any
$t_0 \ge \max_{1\le j \le k} t_*^{(j)}$,
the following holds:
let
$m = (\lambda+1+2k)\frac{2-\alpha}{2}$
and assume
$F \in \bigcap_{j=0}^{k} C^{j+1}([0,\infty); H^{k-j,m}(\Omega))$
satisfies
\begin{align}
    \int_{\Omega} a(x)^{-1} |\partial_t^j F(x,t)|^2 \Psi(x,t;t_0)^{\lambda+1+2j} \,dx \in L^1(0,\infty)
\end{align}
for $j =0,1, \ldots, k$.
Let $w$ be the solution of \eqref{eq:gDW}
in Theorem \ref{thm:reg:gDW}
with the initial data
$(w_0, w_1) \in H^{k+2,m}(\Omega) \times H^{k+1,m}(\Omega)$
satisfying
the $k$-th order compatibility condition
in the sense of Theorem \ref{thm:reg:gDW}.
Then, we have
\begin{align}
    &E^{(j)} [\partial_t^j w](t;t_0,\lambda,\nu^{(j)}) \in L^{\infty}(0,\infty),\\
    &\int_{\Omega} a(x) |\partial_t^{j+1} w(x,t)|^2 \Psi(x,t;t_0)^{\lambda+1+2j} \,dx \in L^1(0,\infty)
\end{align}
for $j=1,\ldots,k$.
\end{theorem}

The proof of Theorem \ref{thm:en:w:k} is 
based on an induction argument,
which is similar to that of Lemma \ref{lem:v2}.
The main part of the induction argument
is the following lemma.
\begin{lemma}\label{lem:en:w:Ej}
Let
$j \in \mathbb{N}$.
Assume $a(x)$ satisfies \eqref{a}.
Let
$\lambda \ge 0$
and
$m = (\lambda+1+2j)(2-\alpha)/2$.
Then, there exist constants
$\nu^{(j)} = \nu^{(j)}(N,\alpha,\lambda,j) >0$
and
$t_*^{(j)} = t_*^{(j)}(N,\alpha,\lambda,\nu^{(j)}, j) \ge 1$
such that for any
$t_0 \ge t_*^{(j)}$,
the following holds:
Assume
$F \in C^1([0,\infty); H^{0,m}(\Omega))$
and let
$w$
be the solution of \eqref{eq:gDW}
with initial data
$(w_0, w_1)\in (H^{2,m}(\Omega)\cap H^{1,m}_0(\Omega)) \times H^{1,m}_0(\Omega)$.
If
\begin{align}
    &\int_{\Omega} a(x)^{-1} |F (x,t)|^2 \Psi(x,t;t_0)^{\lambda+1+2j} \,dx \in L^1(0,\infty), \\
    & \int_{\Omega} a(x) |w(x,t)|^2 \Psi(x,t;t_0)^{\lambda-1+2j} \,dx \in L^1(0,\infty)
\end{align}
are satisfied, then
\begin{align}
    &E^{(j)}[w](t;t_0,\lambda)\in L^{\infty}(0,\infty),\\
    &\int_{\Omega} a(x) |\partial_t w(x,t)|^2 \Psi(x,t;t_0)^{\lambda+1+2j} \,dx \in L^1(0,\infty)
\end{align}
hold.
\end{lemma}
For the proof of Lemma \ref{lem:en:w:Ej},
we further prepare
the following two lemmas.
\begin{lemma}\label{lem:en:w:E0j}
Under the assumptions of Lemma \ref{lem:en:w:Ej},
we have
\begin{align}
    \frac{d}{dt} E_0^{(j)}[w](t;t_0,\lambda)
    &\le
    C \int_{\Omega} |\partial_t w|^2 \Psi^{\lambda + 2j} \,dx
    - \int_{\Omega} |\nabla w (x,t)|^2 \Psi(x,t;t_0)^{\lambda+2j} \,dx \\
    &\quad
    + C \int_{\Omega} a(x) |w(x,t)|^2 \Psi(x,t;t_0)^{\lambda -1 +2j} \,dx \\
    &\quad
    +C \int_{\Omega} a(x)^{-1} |F(x,t)|^2 \Psi(x,t;t_0)^{\lambda+1+2j} \,dx
\end{align}
with some constant
$C>0$.
\end{lemma}
\begin{proof}
We compute
\begin{align}
    \frac{d}{dt} E_0^{(j)}[w](t;t_0,\lambda)
    &= 2 \int_{\Omega} | \partial_t w|^2 \Psi^{\lambda+2j} \,dx
        + 2 \int_{\Omega} w \left(\partial_t^2 w + a(x) \partial_t w \right) \Psi^{\lambda+2j} \,dx \\
    &\quad + 2 (\lambda + 2j) \int_{\Omega}
        \left(
            2 w \partial_t w + a(x) |w|^2
        \right) \Psi^{\lambda-1+2j} \,dx \\
    &= 2 \int_{\Omega} | \partial_t w|^2 \Psi^{\lambda+2j} \,dx
        + 2 \int_{\Omega} w \left( \Delta w + F \right) \Psi^{\lambda+2j} \,dx \\
    &\quad + 2(\lambda + 2j) \int_{\Omega}
        \left(
            2 w \partial_t w + a(x) |w|^2
        \right) \Psi^{\lambda-1+2j} \,dx \\
    &= 2 \int_{\Omega} | \partial_t w|^2 \Psi^{\lambda+2j} \,dx
        - 2 \int_{\Omega} |\nabla w|^2 \Psi^{\lambda + 2j} \,dx \\
    &\quad
        - 2(\lambda+2j ) \int_{\Omega} w ( \nabla w \cdot \nabla \Psi )
            \Psi^{\lambda -1 +2j} \,dx \\
    &\quad + 2(\lambda + 2j) \int_{\Omega}
        \left(
            2 w \partial_t w + a(x) |w|^2
        \right) \Psi^{\lambda -1 +2j} \,dx \\
    &\quad + 2 \int_{\Omega} w F \Psi^{\lambda+2j} \,dx.
\end{align}
The Schwarz inequality implies
\begin{align}
    &\left| 2(\lambda+2j ) \int_{\Omega} w ( \nabla w \cdot \nabla \Psi )
            \Psi^{\lambda -1 +2j} \,dx \right| \\
    &\le \int_{\Omega} |\nabla w|^2 \Psi^{\lambda +2j} \,dx
    + C \int_{\Omega} a(x) |w|^2 \frac{|\nabla \Psi|^2}{a(x) \Psi} \Psi^{\lambda -1 + 2j} \,dx \\
    &\le \int_{\Omega} |\nabla w|^2 \Psi^{\lambda +2j} \,dx
    + C \int_{\Omega} a(x) |w|^2 \Psi^{\lambda -1+2j} \,dx,
\end{align}
where we have used
$\nabla \Psi = \nabla A_{\varepsilon}(x)$,
$\Psi(x,t;t_0) \ge A_{\varepsilon}(x)$,
and \eqref{A3} in Lemma \ref{lem_A_ep}.
Similarly, we have
\begin{align}
    &\left| 2(\lambda + 2j) \int_{\Omega} 2 w \partial_t w \Psi^{\lambda-1+2j} \,dx \right| \\
    &\le C \int_{\Omega} |\partial_t w|^2 \Psi^{\lambda + 2j} \,dx
    + C \int_{\Omega} a(x) |w|^2 a(x)^{-1} \Psi^{\lambda -2+2j} \,dx \\
    &\le C \int_{\Omega} |\partial_t w|^2 \Psi^{\lambda + 2j} \,dx
    + C \int_{\Omega} a(x) |w|^2 \Psi^{\lambda -1 +2j} \,dx,
\end{align}
and
\begin{align}
    &\left| 2 \int_{\Omega} w F \Psi^{\lambda+2j} \,dx \right| \\
    &\le C \int_{\Omega} a(x) |w|^2 a(x)^{\lambda-1+2j} \,dx
    + C \int_{\Omega} a(x)^{-1} |F|^2 \Psi^{\lambda+1+2j} \,dx.
\end{align}
This completes the proof.
\end{proof}
\begin{lemma}\label{lem:en:w:E1j}
Under the assumptions of Lemma \ref{lem:en:w:Ej},
there exists
$t_2 \ge 1$
such that for any $t_0 \ge t_2$ and $t>0$,
we have
\begin{align}
    E_1^{(j)}[w](t;t_0,\lambda)
    &\le - \int_{\Omega} a(x) |\partial_t w(x,t)|^2 \Psi(x,t;t_0)^{\lambda+1+2j} \,dx \\
    &\quad
    + C \int_{\Omega} |\nabla w(x,t)|^2 \Psi(x,t;t_0)^{\lambda+2j} \,dx \\
    &\quad
    + C \int_{\Omega} a(x)^{-1} |F(x,t)|^2 \Psi(x,t;t_0)^{\lambda+1+2j} \,dx
\end{align}
with some constant
$C=C(N,\alpha,\delta,\varepsilon,\lambda,j,t_2)>0$.
\end{lemma}
The proof is completely the same as that of
Lemma \ref{lem:en:w:E1}
by replacing $\lambda$ by $\lambda+2j$.
Thus, we omit the detail.
\begin{proof}[Proof of Lemma \ref{lem:en:w:Ej}]
By Lemmas \ref{lem:en:w:E0j} and \ref{lem:en:w:E1j},
taking $\nu^{(j)} > 0$ sufficiently small,
and then, taking
$t_*^{(j)} \ge t_2$
sufficiently large depending on
$\nu^{(j)}$,
we have
\begin{align}
    \frac{d}{dt} E^{(j)} [w] (t;t_0,\lambda,\nu^{(j)})
    &= \nu^{(j)} \frac{d}{dt} E_1^{(j)} [w] (t;t_0,\lambda)
        + \frac{d}{dt} E_0^{(j)} [w] (t;t_0,\lambda) \\
    &\le
    - \eta_1 \int_{\Omega} a(x) |\partial_t w|^2 \Psi^{\lambda+1+2j} \,dx
    - \eta_2 \int_{\Omega} |\nabla w|^2 \Psi^{\lambda + 2j} \,dx \\
    &\quad
    +  C \int_{\Omega} a(x) |w|^2 \Psi^{\lambda-1+2j} \,dx
    + C \int_{\Omega} a(x)^{-1} |F|^2 \Psi^{\lambda+1+2j} \,dx
\end{align}
for $t_0 \ge t_*^{(j)}$ and $t>0$
with some constants
$\eta_1, \eta_2 > 0$.
By integrating the above inequality on
$[0,t]$
and using the assumptions, we conclude
\begin{align}
    &E^{(j)} [w](t;t_0,\lambda,\nu^{(j)})
    + \int_0^t \int_{\Omega}
        a(x) |\partial_t w|^2 \Psi^{\lambda+1+2j} \,dx d\tau
    + \int_0^t \int_{\Omega}
        |\nabla w|^2 \Psi^{\lambda + 2j} \,dx d\tau \\
    &\le
    E^{(j)} [w] (0;t_0,\lambda,\nu^{(j)})
    + C \int_0^t \int_{\Omega} a(x) |w|^2 \Psi^{\lambda-1+2j} \,dx d\tau \\
    &\quad + C \int_0^t \int_{\Omega} a(x)^{-1} |F|^2 \Psi^{\lambda+1+2j} \,dx d\tau,
\end{align}
and the proof is complete.
\end{proof}
\begin{proof}[Proof of Theorem \ref{thm:en:w:k}]
By Theorem \ref{thm:en:w:1},
there exist
$\nu>0$ and $t_* \ge 1$
such that
\begin{align}
    &E[w] (t;t_0,\lambda, \nu) \in L^{\infty}(0,\infty),\\
\label{eq:en:dw:dtw:L1:1}
    &\int_{\Omega} a(x) |\partial_t w(x,t)|^2 \Psi(x,t;t_0)^{\lambda+1} \,dx \in L^1(0,\infty)
\end{align}
hold for
$t_0 \ge t_*$.
Now, thanks to the property \eqref{eq:en:dw:dtw:L1:1},
we apply
Lemma \ref{lem:en:w:Ej} with
$j=1$
and the replacement of
$w$ and $F$ by $\partial_t w$ and $\partial_t F$,
respectively.
Then, there exist
$\nu^{(1)}>0$
and
$t_*^{(1)} \ge 1$
such that
\begin{align}
    &E^{(1)}[\partial_t w](t;t_0,\lambda,\nu^{(1)}) \in L^{\infty}(0,\infty),\\
    &\int_{\Omega} a(x) |\partial_t^2 w(x,t)|^2 \Psi(x,t;t_0)^{\lambda+1+2} \,dx \in L^1(0,\infty)
\end{align}
hold for
$t_0 \ge t_*^{(1)}$.
The latter property allows us to
apply again Lemma \ref{lem:en:w:Ej} with
$j=2$ and the replacement of
$w$ and $F$ by $\partial_t^2 w$ and $\partial_t^2 F$,
respectively.
Repeating this argument until
$j=k$,
we reach the conclusion of Theorem \ref{thm:en:w:k}.
\end{proof}

\section{Proof of the asymptotic expansion}
In this section,
we give the estimates of the right-hand side of
\eqref{eq:u_decomp} and
complete the proof of Theorem \ref{thm:asym}.

Let $n \in \mathbb{N}$
be fixed,
and let
$\delta \in (0,1/2)$,
$\varepsilon \in (0,1/2)$,
$\lambda \in [0,(1-2\delta)\gamma_{\varepsilon})$,
and
$t_0 \ge 1$.
Moreover, we define
$\lambda_j = \lambda - \frac{2\alpha}{2-\alpha}j$
for $j=1,\ldots,n+1$,
and we assume that
$\lambda_{n+1} \in [0,(1-2\delta)\gamma_{\varepsilon})$,
that is,
$\lambda \in [\frac{2\alpha}{2-\alpha}(n+1), (1-2\delta)\gamma_{\varepsilon})$.
Here, we note that
the assumption
$n+1 < \frac{N-\alpha}{2\alpha}$
ensures that this interval is not empty,
provided that
$\delta$ and $\varepsilon$
are sufficiently small.
We also put
$\tilde{m} := (\lambda_{n+1}+1+2n)\frac{2-\alpha}{2}$.

As in \eqref{ass:thm:asym},
we assume that the initial data
$u_0$ and $u_1$ satisfy
\begin{align}
    u_0 \in H^{s+1,m}(\Omega)\cap H_0^{s,m}(\Omega),\quad
    u_1 \in H^{s,m}_0(\Omega)
\end{align}
with sufficiently large
$s$ and $m$.

Note that, in what follows,
we retake the parameter
$t_0 \ge 1$
suitably larger from line to line.

\textbf{Step 0: Estimates of $V_0$:}
We first give the estimates of
$V_0$,
which is the solution of
\eqref{hx}.
We apply Theorem \ref{thm:v}
with
$k=n$,
$v=V_0$,
$G=0$,
$v_0 = u_0+ a(x)^{-1}u_1$.
Noting Lemma \ref{lem:D(L)}
and taking
$s$ and $m$ sufficiently large,
we have
\begin{align}
   \partial_t^j V_0(x,0)
   &\in
   D(L) \cap
   H^{0, (\lambda+2j) \frac{2-\alpha}{2} -\frac{\alpha}{2}}(\Omega),\\
   \nabla \partial_t^j V_0(x,0)
   &\in H^{0, (\lambda+1+2j) \frac{2-\alpha}{2}}(\Omega)
\end{align}
for $j=0,1,\ldots, n$,
where 
$\partial_t^j V_0(x,0)$
is defined by the right-hand side of
\eqref{eq:dtv0}
with
$v_0 = u_0 + a(x)^{-1}u_1$
and
$G=0$.
Therefore, the assumptions of Theorem \ref{thm:v} are fulfilled, and we have
\begin{align}
\label{eq:est:V0:Linf}
    &\int_{\Omega} a(x) |\partial_t^j V_0(x,t)|^2 \Psi(x,t;t_0)^{\lambda+2j} \,dx \in L^{\infty}(0,\infty),\\
\label{eq:est:V0:L1}
    &\int_{\Omega} a(x) |\partial_t^{j+1} V_0(x,t)|^2 \Psi(x,t;t_0)^{\lambda+1+2j} \,dx \in L^1(0,\infty)
\end{align}
for $j=0,1,\ldots, n$.
In particular,
\eqref{eq:est:V0:Linf} with
$j=0$
implies
the $L^2$-estimate of $V_0$:
\begin{align}
    \left\| V_0(t) \right\|_{L^2(\Omega)}
    \le C (1+t)^{-\frac{\lambda}{2}+\frac{\alpha}{2(2-\alpha)}}.
\end{align}

\textbf{Step 1: Estimates of $V_1$:}
Next, we consider the estimate for $V_1$,
which is the solution of
\eqref{eq:intro:V_j} with $j=1$.
We apply Theorem \ref{thm:v} with
$k=n$,
$v=V_1$,
$v_0 = -a(x)^{-2}u_1(x)$,
$G=-\partial_t V_0$,
and the replacement of
$\lambda$
by
$\lambda_1 = \lambda - \frac{2\alpha}{2-\alpha}$.
Similarly as before,
If we take
$s$ and $m$ sufficiently large,
then we have
\begin{align}
   \partial_t^j V_1(x,0)
   &\in
   D(L) \cap
   H^{0, (\lambda_1+2j) \frac{2-\alpha}{2} -\frac{\alpha}{2}}(\Omega),\\
   \nabla \partial_t^j V_1(x,0)
   &\in H^{0, (\lambda_1+1+2j) \frac{2-\alpha}{2}}(\Omega)
\end{align}
for $j=0,1,\ldots, n$,
where
$\partial_t^j V_1(x,0)$
is defined by the right-hand side of
\eqref{eq:dtv0} with
$v_0 = -(-a(x))^{-2}$
and
$G = -\partial_t V_0$.
Moreover, from \eqref{eq:est:V0:L1},
one obtains
\begin{align}
    &\int_{\Omega} a(x)^{-1} |\partial_t^{j}(-\partial_t V_0)(x,t)|^2 \Psi(x,t;t_0)^{\lambda_1 +1+2j} \,dx \\
    &\le \int_{\Omega} a(x) |\partial_t^{j+1} V_0(x,t)|^2 \Psi(x,t;t_0)^{\lambda+1+2j} \,dx
    \in L^1(0,\infty)
\end{align}
for $j=0,1,\ldots,n$
(this is the reason why we define
$\lambda_1=\lambda-\frac{2\alpha}{2-\alpha}$).
Therefore, the assumptions of Theorem \ref{thm:v}
are fulfilled,
and we deduce
\begin{align}
\label{eq:est:V1:Linf}
    &\int_{\Omega} a(x) |\partial_t^j V_1(x,t)|^2 \Psi(x,t;t_0)^{\lambda_1 + 2j} \,dx \in L^{\infty} (0,\infty),\\
    &\int_{\Omega} a(x) |\partial_t^{j+1} V_1(x,t)|^2 \Psi(x,t;t_0)^{\lambda_1+1+2j} \,dx
    \in L^1(0,\infty)
\end{align}
for $j=0,1,\ldots, n$.
In particular,
\eqref{eq:est:V1:Linf} with $j=1$
implies 
the $L^2$-estimate of
$\partial_t V_1$:
\begin{align}
    \left\| \partial_t V_1 \right\|_{L^2}
    \le C (1+t)^{-\frac{\lambda_1}{2}-1+\frac{\alpha}{2(2-\alpha)}}
    \le C (1+t)^{-\frac{\lambda}{2}- \frac{2(1-\alpha)}{2-\alpha} + \frac{\alpha}{2(2-\alpha)}}.
\end{align}

\textbf{Step $n$: Estimates of $V_n$:}
Continuing this argument until
$j=n$,
we can estimate
$V_n$.
Indeed, 
we apply Theorem \ref{thm:v} with
$k=n$,
$v=V_n$,
$v_0 = -(-a(x))^{-n-1} u_1(x)$,
$G=-\partial_t V_{n-1}$,
and the replacement of
$\lambda$
by
$\lambda_n = \lambda - \frac{2\alpha}{2-\alpha}n$.
Similarly as before,
if we take
$s$ and $m$ sufficiently large,
then we have
\begin{align}
    \partial_t^j V_n(x,0)
    &\in
    D(L) \cap
    H^{0,(\lambda_n+2j)\frac{2-\alpha}{2}-\frac{\alpha}{2}}(\Omega),\\
    \nabla \partial_t^j V_n(x,0)
    &\in H^{0,(\lambda_n+1+2j)\frac{2-\alpha}{2}}(\Omega)
\end{align}
for $j=0,1,\ldots,n$,
where
$\partial_t^j V_n(x,0)$
is defined by the right-hand side of
\eqref{eq:dtv0} with 
$v_0 = -(-a(x))^{-n-1}$
and
$G = -\partial_t V_{n-1}$.
Furthermore, by the $(n-1)$-th step, we have the estimate of inhomogeneous term $-\partial_t V_{n-1}$:
\begin{align}
    &\int_{\Omega} a(x)^{-1} |\partial_t^j (-\partial_t V_{n-1})(x,t)|^2 \Psi(x,t;t_0)^{\lambda_n+1+2j} \,dx \\
    &\le \int_{\Omega} a(x) |\partial_t^{j+1} V_{n-1}(x,t)|^2 \Psi(x,t;t_0)^{\lambda_{n-1}+1+2j} \,dx 
    \in L^1(0,\infty).
\end{align}
Thus, the assumptions of Theorem \ref{thm:v}
are fulfilled, and we have
\begin{align}
\label{eq:Vn:Linfty}
    &\int_{\Omega} a(x) |\partial_t^j V_n(x,t)|^2 \Psi(x,t;t_0)^{\lambda_n + 2j} \,dx
    \in L^{\infty}(0,\infty),\\
\label{eq:Vn:L1}
    &\int_{\Omega} a(x) | \partial_t^{j+1} V_n(x,t)|^2 \Psi(x,t;t_0)^{\lambda_n + 1 + 2j} \,dx \in L^1(0,\infty)
\end{align}
for $j=0,1,\ldots,n$.
In particular,
\eqref{eq:Vn:Linfty} with $j=n$
implies
the $L^2$-estimate of $\partial_t^n V_n$:
\begin{align}\label{eq:Vn:dtn:L2}
    \left\| \partial_t^n V_n(t) \right\|_{L^2(\Omega)}
    \le C (1+t)^{-\frac{\lambda_n}{2} -n +\frac{\alpha}{2(2-\alpha)}}
    = C(1+t)^{-\frac{\lambda}{2} - \frac{2n(1-\alpha)}{2-\alpha} + \frac{\alpha}{2(2-\alpha)}}.
\end{align}

Moreover,
we shall check that
$V_n$ has the regularity
\begin{align}\label{eq:reg:Vn}
    \partial_t V_n \in \bigcap_{j=0}^n C^{j+1} ([0,\infty); H^{n-j,\tilde{m}}(\Omega))
\end{align}
which will be required in the next step.
This follows if both the initial value
$\partial_t V_n(x,0)$ and the inhomogeneous term
$-\partial_t V_{n-1}$ have enough regularity and belong to a suitable weighted Sobolev space.
Note that
$\partial_t V_n(x,0)$
can be computed from \eqref{eq:dtv0},
and the regularity of
$-\partial_t V_{n-1}$
is obtained by a similar iteration argument
as Steps $0,1,\ldots,n$.
Thus, we can obtain \eqref{eq:reg:Vn}
if $s$ and $m$ are sufficiently large.

\textbf{Final step: Estimates of $U_{n+1}$:}
Finally, we estimate
$\partial_t^{n+1}U_{n+1}$
defined by \eqref{eq:intro:U_n+1} with $j=n$.
First, to check the compatibility condition on
$U_{n+1}$,
we prepare the following lemma:
\begin{lemma}\label{lem:compatibility}
Let
$U_{n+1}^{(0)} = 0$,
$U_{n+1}^{(1)} = -(-a(x))^{-n-1}u_1(x)$,
and let
$U_{n+1}^{(p)}$
for $p=2,\ldots,n+1$
be successively defined by
\begin{align}
    U_{n+1}^{(p)}(x) = \Delta U_{n+1}^{(p-2)}
    - a(x) U_{n+1}^{(p-1)}
    - \partial_t^{p-1} V_n(x,0).
\end{align}
Then, we have, for 
$p=2,\ldots, n+1$,
\begin{align}
    U_{n+1}^{(p)}(x)
    = (-a(x))^{-(n-p+2)}u_1(x)
    - \partial_t V_{n-p+2}(x,0)
    - \cdots
    - \partial_t^{p-1} V_n(x,0).
\end{align}
\end{lemma}
\begin{proof}
When $p=2$, the conclusion is obvious.
Suppose that the conclusion is true up to $p-1$, and consider the case $p$.
Using the assumption of the induction, 
we have
\begin{align}
    U_{n+1}^{(p)}(x)
    &=
    \Delta U_{n+1}^{(p-2)}
    - a(x) U_{n+1}^{(p-1)}
    - \partial_t^{p-1} V_n(x,0)\\
    &=
    \Delta \left( (-a(x))^{-(n-p+4)}u_1(x) - \partial_t V_{n-p+4}(x,0) - \cdots
    - \partial_t^{p-3} V_n(x,0) \right) \\
    &\quad -a(x) \left(
    (-a(x))^{-(n-p+3)}u_1(x)
    - \partial_t V_{n-p+3}(x,0)
    - \cdots
    - \partial_t^{p-2} V_n(x,0) \right) \\
    &\quad - \partial_t^{p-1} V_n(x,0).
\end{align}
Recalling
$a(x)\partial_t V_j - \Delta V_j = - \partial_t V_{j-1}$
and
$V_j(x,0) = -(-a(x))^{-j-1}u_1(x)$,
we see that the right-hand side becomes
\begin{align}
    (-a(x))^{-(n-p+2)}u_1(x)
    - \partial_t V_{n-p+2}(x,0)
    - \cdots
    - \partial_t^{p-1} V_n(x,0),
\end{align}
which completes the proof.
\end{proof}

Since
$u_1 \in H^{s,m}_0(\Omega)$
with sufficiently large
$s$ and $m$,
we easily see
$U_{n+1}^{(1)} = - (-a(x))^{-n-1}u_1 \in H^{n+1, \tilde{m}}(\Omega)$
and
\begin{align}
    &(-a(x))^{-(n-p+2)}u_1 \in H^{2,\tilde{m}}(\Omega) \cap H_0^{1,\tilde{m}}(\Omega)
    \quad (p=0,\ldots,n),\\
    &(-a(x))^{-1} u_1 \in H_0^{1,\tilde{m}}(\Omega).
\end{align}
Moreover, from \eqref{eq:dtv0} and a
similar induction argument,
we can see that the functions
$\partial_t^p V_j (x,0)$
satisfy the following conditions,
provided that
$s$ and $m$ are sufficiently large:
for 
$V_j \ (j=2,\ldots,n)$,
\begin{align}
    \partial_t^p V_j(x,0) &\in H^{2,\tilde{m}}(\Omega) \cap H_0^{1,\tilde{m}}(\Omega)
    \quad (p=1,\ldots,j-1),\\
    \partial_t^j V_j(x,0) &\in H_0^{1,\tilde{m}}(\Omega);
\end{align}
for
$V_1$,
\begin{align}
    \partial_t V_1(x,0) \in H_0^{1,\tilde{m}}(\Omega).
\end{align}
Therefore, by applying Lemma \ref{lem:compatibility},
we see that the data
$U_{n+1}^{(p)}$ for
$p=0,1,\ldots,n+1$
satisfy the $n$-th order compatibility condition
\begin{align}
    &(U_{n+1}^{(p)}, U_{n+1}^{(p+1)})
    \in (H^{2,\tilde{m}}(\Omega)\cap H^{1,\tilde{m}}_0(\Omega))\times H^{1,\tilde{m}}_0(\Omega)
\end{align}
for $p=0,\ldots,n$.
Combining this with the fact
\eqref{eq:reg:Vn},
we can apply
Theorem \ref{thm:reg:gDW}
with $k=n, m= \tilde{m}$
to obtain the regularity of $U_{n+1}$
and the weighted energy of
$\partial_t^{n}U_{n+1}$
is well-defined.
Moreover, it follows from
\eqref{eq:Vn:L1} that
\begin{align}
    \int_{\Omega} a(x)^{-1} |\partial_t^j (-\partial_t V_n(x,t))|^2 \Psi(x,t;t_0)^{\lambda_{n+1}+1+2j} \,dx
    \in L^1(0,\infty)
\end{align}
for $j=0,1,\ldots,n$.
Hence, we can apply
Theorems \ref{thm:en:w:1} and \ref{thm:en:w:k},
and there exist
$\nu^{(n)} > 0$ and $t_*^{(n)} \ge 1$
such that for any $t_0 \ge t_*^{(n)}$
we have
\begin{align}
    E^{(n)}[\partial_t^n U_{n+1}](t;t_0,\lambda_{n+1},\nu^{(n)}) \in L^{\infty}(0,\infty).
\end{align}
In particular,
the above bound yields
\begin{align}
    \int_{\Omega} |\partial_t^{n+1} U_{n+1}(x,t)|^2 \,dx
    &\le C (1+t)^{-\lambda_{n+1}-1-2n}.
\end{align}
Namely, we conclude
\begin{align}
    \left\| u(t) - \sum_{j=0}^{n} \partial_t^j V_j(t) \right\|_{L^2(\Omega)}
    &=
    \| \partial_t^{n+1} U_{n+1}(t)\|_{L^2(\Omega)} \\
    &\le
    C (1+t)^{-\frac{\lambda}{2} - \frac{(2n+1)(1-\alpha)}{2-\alpha}
    +\frac{\alpha}{2(2-\alpha)}},
\end{align}
which completes the proof of
Theorem \ref{thm:asym}.

\section*{Acknowledgements}
This work was supported by JSPS KAKENHI Grant Numbers
JP16K17625,
JP18H01132,
18K13445,
JP20K14346.

\end{document}